\documentclass[11pt]{amsart}
\usepackage{amssymb,amsmath}
\usepackage{graphicx}
\usepackage{subcaption}
\usepackage{bm,bbm}
\usepackage{color}
\usepackage{tikz}
\usetikzlibrary{snakes}
\usepackage{pgfplots}
\usepackage{array,blkarray}
\usepackage[hidelinks]{hyperref}
\usepackage{mathtools}  
\newtheorem{theorem}{Theorem}

\theoremstyle{definition}

\newtheorem{example}[theorem]{Example}
\theoremstyle{lemma}
\newtheorem{lemma}[theorem]{Lemma}

\theoremstyle{remark}
\newtheorem{remark}[theorem]{Remark}
\newtheorem{assumption}[theorem]{Assumption}
\numberwithin{theorem}{section}
\numberwithin{equation}{section}
\numberwithin{table}{section}
\numberwithin{figure}{section}
\def\R{\mathbb{R}}
\def\C{\mathbb{C}}
\definecolor{myBlue2}{RGB}{113,104,238} 
\definecolor{myBlue3}{RGB}{30,144,255} 
\definecolor{myGreen2}{RGB}{69,169,0} 
\definecolor{myGreen3}{RGB}{154,205,50} 
\definecolor{myRed2}{RGB}{165,42,42} 

\definecolor{color0}{rgb}{0.12156862745098,0.466666666666667,0.705882352941177}
\definecolor{color1}{rgb}{1,0.498039215686275,0.0549019607843137}
\definecolor{color3}{rgb}{0.83921568627451,0.152941176470588,0.156862745098039}
\definecolor{color2}{rgb}{0.172549019607843,0.627450980392157,0.172549019607843}

\DeclareMathOperator{\sspan}{span}

\DeclareMathOperator{\range}{im}

\DeclareMathOperator{\Res}{Res}

\DeclareMathOperator{\real}{Re}

\newcommand{\norm}[1]{\Vert #1 \Vert }

\newcommand{\calA}{\ensuremath{\mathcal{A}}}
\newcommand{\calB}{\ensuremath{\mathcal{B}}}

\newcommand{\calF}{\ensuremath{\mathcal{F}}}
\newcommand{\calH}{\ensuremath{\mathcal{H}}}
\newcommand{\calI}{\ensuremath{\mathcal{I}}}
\newcommand{\calK}{\ensuremath{\mathcal{K}}}
\newcommand{\cL}{\ensuremath{\mathcal{L}}}
\newcommand{\calL}{\cL}

\newcommand{\calO}{\mathcal{O}}
\newcommand{\calP}{\ensuremath{\mathcal{P}}}

\newcommand{\calQ}{\ensuremath{\mathcal{Q}}}

\newcommand{\calT}{\ensuremath{\mathcal{T}}}

\newcommand{\calV}{\ensuremath{\mathcal{V}}}

\newcommand{\y}{\ensuremath{y}}
\renewcommand{\b}{\ensuremath{\y}}

\newcommand{\G}{\ensuremath{G}}

\def\dx{\,\text{d}x}

\newcommand{\hook}{\hookrightarrow}
\newcommand{\conv}{\rightarrow}
\newcommand{\wconv}{\rightharpoonup}
\def\ab{a_\beta}
\def\calAb{\calA_\beta}
\def\k{{\mathbf k}}
\newcommand{\Ak}{\ensuremath{\calA_{\k}}}       
\newcommand{\Aktb}{\ensuremath{\calA_{\k,\tbeta}}} 
\newcommand{\calIa}{\ensuremath{\calI_{\alpha}}}
\newcommand{\VxC}{\ensuremath{\mathcal{W}}}
\newcommand{\LOmgII}{\ensuremath{X}}
\newcommand{\oI}{\ensuremath{\overline{\calI}_2}}
\newcommand{\aVec}{\ensuremath{\mathbf{a}} }
\newcommand{\bVec}{\ensuremath{\mathbf{b}} }
\newcommand{\cVec}{\ensuremath{\mathbf{c}} }
\newcommand{\dVec}{\ensuremath{\mathbf{d}} }

\newcommand{\xVec}{\ensuremath{\boldsymbol{x}} }
\newcommand{\yVec}{\ensuremath{\boldsymbol{y}} }
\newcommand{\zVec}{\ensuremath{\boldsymbol{z}} }

\newcommand{\AMat}{\ensuremath{\mathbf{A}} }

\newcommand{\IMat}{\ensuremath{\mathbf{I}} }
\newcommand{\bigAMat}{\ensuremath{\mathbb{A}}}
\newcommand{\bigIMat}{\ensuremath{\mathbb{I}}}
\newcommand{\bigAForm}{\ensuremath{\mathfrak{a}}}
\newcommand{\bigIForm}{\ensuremath{\mathfrak{i}}}
\newcommand{\aalpha}{\ensuremath{\underline{\alpha}} }
\newcommand{\tbeta}{\ensuremath{\tilde\beta} }
\def\V{\calV}
\def\cH{\calH}
\def\Hper{H^1_\text{per}(\Omega)}
\newcommand{\krylov}{\ensuremath{\mathfrak{K}}}
\def\eps{\varepsilon}
\def\cell{c_{\text{ell}}}

\newcommand\ulim{u^{\scalebox{0.5}{$\bigstar$}}}
\newcommand\ulimt{{\tilde u}^{\scalebox{0.5}{$\bigstar$}}}
\newcommand\vlim{v^{\scalebox{0.5}{$\bigstar$}}}
\newcommand\llim{\lambda^{\scalebox{0.5}{$\bigstar$}}}
\newcommand\omlim{\omega^{\scalebox{0.5}{$\bigstar$}}}
\newcommand\mulim{\mu^{\scalebox{0.5}{$\bigstar$}}}
\newcommand\muref{\mu^{\scalebox{0.5}{\text{ref}}}}
\newcommand\xlim{{\xVec}^{\scalebox{0.5}{$\bigstar$}}}
\newcommand\zlim{{\zVec}^{\scalebox{0.5}{$\bigstar$}}}
\newcommand\zlimt{{\tilde \zVec}^{\scalebox{0.5}{$\bigstar$}}}
\newcommand{\bb}{[\b]}
\newcommand{\blueCirLight}{%
	\begin{tikzpicture}[inner sep=0pt, baseline=(base)]%
	\draw[color0, dotted, thin](0,0) -- (5mm,0); 
	\node[color0, mark size=2.5, opacity=0.5] at (2.5mm,0){%
		\pgfuseplotmark{*}%
	};
	\node[color0, mark size=2.5] at (2.5mm,0){%
		\pgfuseplotmark{o}%
	};
	\node (base) at (0,-.6ex) {};
	\end{tikzpicture}%
}
\newcommand{\orangeTriLight}{%
	\begin{tikzpicture}[inner sep=0pt, baseline=(base)]%
	\draw[color1, dotted, thin](0,0) -- (5mm,0); 
	\node[color1, mark size=3.2, opacity=0.5] at (2.5mm,0){%
		\pgfuseplotmark{triangle*}%
	};
	\node[color1, mark size=3.2] at (2.5mm,0){%
		\pgfuseplotmark{triangle}%
	};
	\node (base) at (0,-.6ex) {};
	\end{tikzpicture}%
}
\newcommand{\greenSquLight}{%
	\begin{tikzpicture}[inner sep=0pt, baseline=(base)]%
	\draw[color2, dotted, thin](0,0) -- (5mm,0); 
	\node[color2, mark size=2.5, opacity=0.5] at (2.5mm,0){%
		\pgfuseplotmark{square*}%
	};
	\node[color2, mark size=2.5, ] at (2.5mm,0){%
		\pgfuseplotmark{square}%
	};
	\node (base) at (0,-.6ex) {};
	\end{tikzpicture}%
}
\usepackage{etoolbox}
\makeatletter
\patchcmd{\@maketitle}
  {\ifx\@empty\@dedicatory}
  {\ifx\@empty\@date \else {\vskip3ex \centering\footnotesize\@date\par\vskip1ex}\fi
   \ifx\@empty\@dedicatory}
  {}{}
\patchcmd{\@adminfootnotes}
  {\ifx\@empty\@date\else \@footnotetext{\@setdate}\fi}
  {}{}{}
\makeatother

\begin{document}
\newcommand{\myTitleI}{PDE Eigenvalue Iterations with Applications in}
\newcommand{\myTitleII}{Two-dimensional Photonic Crystals}
\newcommand{\myTitle}{\Large \myTitleI\\[0.3em] \myTitleII}
\title[PDE eigenvalue iterations for photonic crystals]{\myTitle$^*$}
\author[]{R.~Altmann$^\dagger$, M.~Froidevaux$^{\ddagger}$}
\address{${}^{\dagger}$ Institut f\"ur Mathematik, Universit\"at Augsburg, Universit\"atsstr.~14, 86159 Augsburg, Germany}
\address{${}^{\ddagger}$ Institut f\"ur Mathematik MA4-5, Technische Universit\"at Berlin, Stra\ss e des 17.~Juni 136, 10623 Berlin, Germany}
\email{robert.altmann@math.uni-augsburg.de, froideva@math.tu-berlin.de}
\thanks{${}^{*}$ Research funded by the Einstein Foundation Berlin in the frame of Einstein Center for Mathematics Berlin ECMath via project OT10 {\em Model Reduction for Nonlinear Parameter-Dependent Eigenvalue Problems in Photonic Crystals} and by
the Deutsche Forschungsgemeinschaft (DFG, German Research Foundation) under Germany's Excellence Strategy -- The Berlin Mathematics Research Center MATH+ (EXC-2046/1, project ID: 390685689).}
\date{\today}
\keywords{}
%
\begin{abstract}
We consider PDE eigenvalue problems as they occur in two-dimensional photonic crystal modeling. If the permittivity of the material is frequency-dependent, then the eigenvalue problem becomes nonlinear. In the lossless case, linearization techniques allow an equivalent reformulation as an extended but linear and Hermitian eigenvalue problem, which satisfies a G\aa{}rding inequality. For this, known iterative schemes for the matrix case such as the inverse power or the Arnoldi method are extended to the infinite-dimensional case. We prove convergence of the inverse power method on operator level and consider its combination with adaptive mesh refinement, leading to substantial computational speed-ups. For more general photonic crystals, which are described by the Drude-Lorentz model, we propose the direct application of a Newton-type iteration. Assuming some a priori knowledge on the eigenpair of interest, we prove local quadratic convergence of the method. 
Finally, numerical experiments confirm the theoretical findings of the paper. 
\end{abstract}
%
\maketitle
%
{\tiny{\bf Key words.} nonlinear eigenvalue problem, photonic crystals, inverse power method, Newton iteration}\\
\indent
{\tiny{\bf AMS subject classifications.} {\bf 65N25}, {\bf 65J10}, {\bf 65F15}} 
%
%
\section{Introduction}
Eigenvalue problems including partial differential equations (PDE) appear in several applications such as structural mechanics~\cite{BatW73}, fluid-solid structures~\cite{Vos03}, or the simulation of Bose-Einstein condensates~\cite{PitS03}. In general, such problems are considered in order to optimize certain properties or parameters of the underlying dynamical system~\cite{MehV04}. 
%
In this paper, we focus on applications as they appear in the modeling of photonic crystals~\cite{Joh87,Kuc01}. These are special composite materials with a periodic structure that affect the propagation of electromagnetic waves and thus, can be used for trapping and guiding light. As these crystals can be designed and manufactured for industrial applications, the aim is to find so-called {\em photonic band-gaps}, which prevent light within a specified frequency range from propagating~\cite{JoaJWM08, Joh12}. Direct applications areas are optical fibers~\cite[Ch.~5]{GonH14}, medical technologies with laser guides for cancer surgeries~\cite{Tsa12}, and thin film solar cells~\cite{DemJ12}. 

The corresponding mathematical model is given by a sequence of nonlinear PDE eigenvalue problems based on the Maxwell equations~\cite{SpeP05,DoeLPSW11}. An important role is played by the {\em electric permittivity}~$\eps$, which is periodic in space and characterizes certain properties of the crystal. If~$\eps$ is independent of the frequency, then we obtain a linear eigenvalue problem. In more realistic models, however, the permittivity is approximated by a rational function, which introduces the nonlinearity into the eigenvalue problem.  \medskip

Numerical methods for computing the spectrum of such materials have been studied intensively. This includes adaptive finite element methods~\cite{BurKSSZ06,GiaG12}, Newton-type methods~\cite{Kre09,HuaLM16}, and linearization techniques~\cite{SuB11,EffKE12}. In the latter case, a spatial discretization is assumed and yields then a linear but extended eigenvalue problem, for which well-known iteration schemes can be applied. 
The combination of mesh refinement and (inexact) eigenvalue iteration methods has been considered in \cite{MehM11, Mie11}.

Corresponding iterative methods for the operator case have, so far, not received much attention in the literature. In the first part of this paper, we focus on linear and Hermitian PDE eigenvalue problems in the weak formulation. This corresponds to the case of a frequency-independent permittivity. Convergence of the (inverse) power method for compact operators mapping from a Hilbert space $\cH$ to $\cH$ was already shown in~\cite{EriSL95}. In this setting, the proof basically follows the same lines as in the finite-dimensional case. General bounded operators were considered in~\cite{EasE07} but only together with a power iteration based on the exact (und thus unknown) eigenvalue. 
Considering the weak formulation, which is more natural in view of spatial discretization methods, we are in a different setting. Nevertheless, the power method converges if an appropriate scaling is included. For the $p$-Laplacian eigenvalue problem, this was shown in~\cite{Boz16}. 
Note that the proven convergence on operator level has the advantage that it is mesh- and basis-independent, cf.~\cite{AltHP19ppt} for a similar approach applied to nonlinear eigenvector problems. Thus, dependencies on a particular discretization do not play a role. Further, the approach allows to choose the spatial discretization in each iteration step independently. 	

The second part of the paper focuses on the nonlinear case as it appears in two-dimensional photonic crystal modeling. Here we consider two different paths of either linearizing the problem or applying directly a Newton iteration. In the first case, we adapt the techniques introduced in~\cite{SuB11,EffKE12} in combination with an inverse power iteration applied to the resulting linear problem. 
Due to the linearization, certain compactness properties get lost, which calls for a novel convergence analysis for the inverse power method. 
The second strategy translates the local convergence of Newton's method from~\cite{Sch08} to the operator case. An analogous method for infinite-dimensional eigenvalue problems was developed in~\cite{AnsR68}, and its local convergence was proven for Fredholm operators with index~$0$ mapping from a Hilbert space $\cH$ to $\cH$. We show a similar result when the operators arise from the here considered weak formulation. \medskip

The paper is structured as follows. In Section~\ref{sect:formulation} we introduce the problem setting, i.e., the linear PDE eigenvalue problem in its weak and operator formulation. Here we gather all the assumptions on the spaces and included operators. In particular, we assume an underlying Gelfand triple with a compact embedding. Section~\ref{sect:convHerm} then considers several iteration schemes including the inverse power method and the Arnoldi method. 
Two-dimensional photonic crystals with frequency-dependent permittivity are then topic of Sections~\ref{sect:nonlinearHermit} and~\ref{sect:nonlinear}. First, we consider a special Hermitian case. For this, we apply a linearization and the inverse power method. More realistic models are then discussed in Section~\ref{sect:nonlinear}, for which we prove the local convergence of Newton's method. Finally, we present three numerical examples in Section~\ref{sect:numerics}.
%
%
\section{Preliminaries}\label{sect:formulation}
As described in the introduction, we consider the weak formulation of a PDE eigenvalue problem. Given the sesquilinear forms~$a\colon \V\times\V \to \C$ and~$(\cdot\,, \cdot)\colon \V \times \V \to \C^+$, we search for a non-trivial pair~$(u, \lambda) \in \V\times \C$ such that for all test functions $v\in \V$ it holds that 
\begin{align}
\label{eq:evp:weak}
  a(u, v) = \lambda\, (u,v). 
\end{align}
More precisely, considering Hermitian eigenvalue problems, we are interested in the eigenpair corresponding to the smallest eigenvalue. In the following, we gather assumptions on the space $\V$ and the included sesquilinear forms. Afterwards we discuss well-known PDE eigenvalue problems, which fit into the given framework. 
%
%
\subsection{General setting}\label{sect:formulation:setting}
We start with general assumptions on the involved function spaces. 
\begin{assumption}[Function spaces]
\label{ass_spaces}
We assume $\V$ to be a complex, separable, and reflexive Banach space. Furthermore, we assume the existence of a complex and separable Hilbert space~$\cH$ (the pivot space) such that $\V$, $\cH$, $\V^*$ form a {\em Gelfand triple}, cf.~\cite[Ch.~23.4]{Zei90a} and \cite[Ch.~11.4]{Bre10}. This means, in particular, that the embedding $i_{\V\hook \cH}\colon \V \hook \cH$ is continuous and dense \cite[Ch.~17.1]{Wlo87}. The continuity constant is denoted by $C_{\V\hook \cH}$.
\end{assumption}
\begin{assumption}[Compactness]
\label{ass_compact}
The embedding $\V \hook \cH$ is compact.
\end{assumption}
With the pivot space $\cH$ in hand, we assume that the sesquilinear form~$(\cdot\,, \cdot)$ in~\eqref{eq:evp:weak} is also defined for functions in $\cH$. We even assume that this defines the inner product in the Hilbert space~$\cH$ and set $\Vert\cdot\Vert := \Vert \cdot \Vert_{\cH} = (\cdot\,, \cdot)^{1/2}$. Further, $j_\cH\colon \cH \to \cH^*$ denotes the Riesz isomorphism. The norm in the space $\V$ is denoted by $\Vert \cdot \Vert_{\V}$. For the sesquilinear form~$a$ we consider the following assumptions. 
\begin{assumption}[Sesquilinear form]
\label{ass_a}
The sesquilinear form $a\colon \V\times\V \to \C$ is assumed to be continuous and Hermitian such that $a(u,u) \in \R$ for all $u \in \V$. Furthermore, $a$ satisfies a G\aa{}rding inequality, i.e.,
\begin{equation*}
  a(u, u) \ge \alpha\, \Vert u \Vert_\V^2  - \beta\, \Vert u\Vert^2
\end{equation*}
for real constants $\alpha >0$, $\beta \ge 0$ and all $u, v\in \V$, see e.g.~\cite{QuaV94}.
\end{assumption}
\begin{remark}
\label{rem:shift}
The previous Assumption~\ref{ass_a} implies that $\ab(u, v) := a(u, v) + \beta\, (u,v)$ is $\V$-coercive and thus, defines an inner product in $\V$. As a result, $\V$ is actually a Hilbert space and the corresponding norm 
\[
	\Vert u\Vert_\beta 
	:= \ab(u, u)^{1/2} 
	\ge \sqrt{\alpha}\, \Vert u \Vert_\V, 
\]
the so-called energy norm, 
is equivalent to $\Vert \cdot \Vert_\V$. Further, all eigenvalues $\lambda$ of~\eqref{eq:evp:weak} satisfy $\ab(u,u)=(\lambda+\beta)(u,u)$, which implies $\mathbb{R}\ni(\lambda+\beta)>0$ and thus $\mathbb{R}\ni \lambda > -\beta$. 
\end{remark}
In order to be well-posed, an eigenvalue problem of the form~\eqref{eq:evp:weak} requires boundary conditions. Throughout this paper, we assume that these conditions are included in the space $\V$, cf.~the examples in the following subsection. 
We close this preliminary part with the proof of Young's inequality in the specific case of complex vectors. 
\begin{lemma}[Young's inequality]
\label{lem:Young}
Consider $\aVec, \bVec \in \C^2$. Then, for every $\delta>0$ we have an estimate of the form 
\begin{equation*}
	|\aVec \cdot\overline{\bVec}+\bVec\cdot\overline{\aVec}|
	\leq \frac{1}{\delta}\, |\aVec|^2 + \delta\, |\bVec|^2,
\end{equation*}
where $\cdot$ denotes the real dot product.
\end{lemma}
\begin{proof}
For any two vectors $\cVec, \dVec \in \C^2$, the following estimates hold:
\begin{align*}
	0&\leq|\cVec+\dVec|^2=(\cVec+\dVec)\cdot\overline{(\cVec+\dVec)}=|\cVec|^2 +|\dVec|^2+\cVec\cdot\overline{\dVec}+\dVec\cdot\overline{\cVec}, \\
	0&\leq|\cVec-\dVec|^2=(\cVec-\dVec)\cdot\overline{(\cVec-\dVec)}=|\cVec|^2 +|\dVec|^2-\cVec\cdot\overline{\dVec}-\dVec\cdot\overline{\cVec}.
\end{align*}
As a consequence, we have $|\cVec \cdot\overline{\dVec}+\dVec\cdot\overline{\cVec}|\leq |\cVec|^2+|\dVec|^2$. The claim then follows by setting $\cVec=\aVec/\sqrt{\delta}$, and $\dVec=\bVec\sqrt{\delta}$.
\end{proof}
%
%
\subsection{Examples}\label{sect:formulation:exp}
We present a couple of well-known examples, which fit into the given framework if formulated in the weak setting. 
\begin{example}[Laplace eigenvalue problem]
\label{exp:laplace}
Consider the eigenvalue problem $-\Delta u = \lambda u$ in a bounded domain~$\Omega$ with homogeneous Dirichlet boundary conditions. 
For this, we set $\V := H^1_0(\Omega)$ with $\Vert \cdot\Vert_\V := \Vert \nabla\cdot\Vert_{L^2(\Omega)}$ and $\cH := L^2(\Omega)$ with the standard $L^2$ inner product. The corresponding sesquilinear form reads $a(u,v) := \int_{\Omega} \nabla u \cdot\overline{\nabla v}\dx$. Note that this implies $a(u, u) = \Vert u\Vert_\V^2$ and thus, $\alpha=1$ and $\beta=0$. The weak form of the Laplace eigenvalue problem then reads: find a non-trivial pair $(u, \lambda) \in \V\times\C$ such that for all $v\in\V$ it holds that
\[
  a(u,v) = \lambda\, (u, v).
\] 
\end{example}
\begin{example}[Schr\"odinger eigenvalue problem]
\label{exp:schroedinger}
The computation of the ground state of the linear Schr\"odinger operator leads to the sesquilinear form 
\[
	a(u, v) 
	:= \int_{\Omega} \nabla u(x)\cdot \overline{\nabla v(x)} + W(x)\, u(x) \overline{v(x)} \dx
\]
with a real-valued potential $W\in L^\infty(\Omega)$. 
For homogeneous Dirichlet boundary conditions this leads to the same spaces $\V$ and $\cH$ as in Example~\ref{exp:laplace}. Further, $a$ satisfies the G\aa{}rding inequality with $\alpha=1$ and $\beta= \max\big\{0, -\inf_{x\in\Omega} W(x)\big\}$.   
For periodic boundary conditions one has to replace the space $\V$ accordingly. 
\end{example}
In this paper, we focus on applications with photonic crystals. The dynamics of the electromagnetic fields inside such a crystal can be modelled by the Maxwell equations in the whole domain $\R^d$, cf.~\cite[Ch.~1]{DoeLPSW11}. These equations combine the magnitudes of the time-harmonic electric and magnetic fields $E$, $H$ and the frequency~$\omega$, which takes the role of an eigenvalue. 

One crucial parameter within the equations is the relative electric permittivity of the materials inside the crystal. We assume the relative permittivity $\eps$ to be piecewise-constant and periodic in space as well as bounded in the sense that 
\[
	1 
	\le \eps(x, \omega)
	\le \eps_{\max}
	< \infty
\]
for all $x\in \R^d$ and $\omega$ in some frequency domain of interest. For the applications in mind, where~$\eps$ is given as a rational function, this means that we consider~$\omega$ bounded away from the poles.  In the two-dimensional case, i.e., when $\eps$ is periodic within a two-dimensional plane and constant along the direction orthogonal to this plane, the Maxwell eigenvalue problem decouples into so-called \textit{transverse magnetic} (TM) and \textit{transverse electric} (TE) modes. 
Thanks to the periodicity of $\eps$, which implies a discrete translational symmetry in the system, a \textit{Floquet transformation} can be applied to reduce the problem posed in~$\R^2$ to a family of problems on a bounded domain $\Omega$ called the \textit{Wigner--Seitz cell} of the crystal lattice, see e.g.~\cite{Kuc01, DoeLPSW11}. Note that the function $\eps$ is here a unitless quantity expressed relative to $\eps_0$, the vacuum permittivity. Within this paper, we assume the magnetic permeability of the crystal to be constant and equal to that of vacuum denoted by~$\mu_0$. 
\begin{example}[TM mode]
\label{exp:TM}
In the two-dimensional setting we consider the TM mode with a real-valued frequency-independent function $\eps(x)$. The resulting PDE eigenvalue problem describes the third component of the electric field $E_3$, from which one can directly compute the components $H_1$ and $H_2$. 
Let $\k$ be a fixed wave vector in the so-called {\em irreducible Brillouin zone} $\calK \subset \R^2$, cf.~\cite[Ch.~1]{DoeLPSW11}, $u_\k$ be the Floquet transform of $E_3(x)$ at $\k$, and $\nabla_\k\coloneqq \nabla+i\k$ denote the {\em shifted gradient}. Then, $u_\k$ satisfies the eigenvalue problem 
\[
	-\nabla_\k \cdot \nabla_\k\, u_\k(x) = \omega^2 \mu_0\ \eps_0\, \eps(x) u_\k(x)
\]
for all $x\in\Omega$.
For the sake of conciseness, we use in the following by abuse of notation a scaled frequency defined by $\omega\to \omega/\sqrt{\mu_0 \eps_0}$. For the weak formulation of the eigenvalue problem we then define $\lambda\coloneqq \omega^2$. Including periodic boundary conditions, we set $\V = \Hper$ with the standard $H^1$-norm and $\cH= L^2(\Omega)$. Note that $\V$ is densely embedded in $\cH$ and thus, Assumption~\ref{ass_spaces} is satisfied, cf.~\cite[Ch.~4.4]{Bre10}. The sesquilinear form $a$ and the (weighted) inner product in $\cH$ then read
\begin{equation}
\label{eq:formA_TM}
 	a(u, v) 
	:= \int_{\Omega} \nabla_\k u(x)\cdot \overline{\nabla_\k v(x)} \dx
\end{equation}
and
\[
	(u, v) := \int_{\Omega}\eps(x) u(x) \overline{v(x)} \dx.
\]
Due to the boundedness of~$\eps$, the sesquilinear form~$(\cdot, \cdot)$ defines an inner product in~$\cH$. The following lemma shows that $a$ satisfies Assumption~\ref{ass_a}. 
\end{example}
\begin{lemma}
\label{lem:betaTM}
For a fixed wave vector $\k \in \R^2$ the sesquilinear form $a$ defined in~\eqref{eq:formA_TM} in Example~\ref{exp:TM} is Hermitian, continuous, and satisfies G{\aa}rding's inequality for any $\beta>0$. 
\end{lemma}
\begin{proof}
Clearly, the sesquilinear form $a$ is Hermitian. For the continuity we apply the Cauchy-Schwarz inequality with respect to the complex dot product in $\C^2$ as well as the inner product in $\cH$ and obtain for all $u, v \in \V$,
\[
	a(u,v)
	\le \int_{\Omega} |\nabla_\k u | |\nabla_\k v| \dx 
	\le \Big(\int_{\Omega} |\nabla_\k u |^2 \dx \Big)^{1/2}  \Big(\int_{\Omega} | \nabla_\k v |^2 \dx\Big)^{1/2}.
\]
Young's inequality from Lemma~\ref{lem:Young} with $\delta=1$ then yields
\begin{align*}
	\int_{\Omega} |\nabla_\k u |^2 \dx 
	&= \int_{\Omega} |\nabla u |^2 + |\k|^2 |u|^2 + \nabla u \cdot\overline{(i \k u)} + {(i \k u)}\cdot \overline{\nabla u} \dx \\
	&\leq \int_{\Omega} |\nabla u|^2 + |\k|^2|u|^2 + |\nabla u|^2 +|\k|^2|u|^2 \dx
	\leq 2\max\{1, |\k|^2 \} \Vert u\Vert_\V^2, 
\end{align*}
which proves the continuity of $a$. To show the G{\aa}rding inequality, we first consider the case $\k=0$. Then, for any $0<\beta \le 1$, we have 
\[
	a(u,u)
	= \int_{\Omega} |\nabla u |^2 \dx
	\ge (1-\beta) \Vert \nabla u \Vert^2_{L^2(\Omega)} + \beta\, \Vert u \Vert_\V^2 - \beta \int_{\Omega} \eps\, u \overline{u} \dx
	\ge \beta \Vert u \Vert_\V^2 - \beta \Vert u\Vert^2,
\]
i.e., the G{\aa}rding inequality with~$\alpha=\beta$. Otherwise, for $\k\neq 0$, we apply once more Lemma~\ref{lem:Young} for some parameter $\delta > 0$ and get 
\begin{align*}
	a(u,u)
	= \int_{\Omega} |\nabla_\k u |^2 \dx 	
	&\geq \int_{\Omega} |\nabla u|^2+|\k|^2 |u|^2 - \big| \nabla u \cdot \overline{(i\k u)} + (i\k u)\cdot \overline{\nabla u} \big| \dx \\
	&\geq \int_{\Omega} |\nabla u|^2+|\k|^2 |u|^2  - \frac{|\nabla u|^2}{\delta} - \delta |\k|^2|u|^2 \dx. 
\end{align*}
%
%
Now assume $\delta>1$ and define $\alpha\coloneqq (1- \delta^{-1})\min\{1, |\k|^2 \}>0$. Then, $a$ satisfies the G{\aa}rding inequality with $\beta\coloneqq (\delta-\delta^{-1})\, |\k|^2>0$. Note that $\beta$ can be chosen arbitrarily small with an appropriate choice of $\delta>1$. 
\end{proof}
\begin{example}[TE mode]
\label{exp:TE}
In the eigenvalue problem corresponding to the TE mode, the relative permittivity $\eps$ appears in the differential operator. More precisely, for a fixed wave vector $\k$, we search for $u_\k$ such that 
\[
	-\nabla_\k \cdot \frac{1}{\eps(x)} \nabla_\k\, u_\k(x)
	= \lambda\, u_\k(x)
\]
for all $x\in\Omega$. Considering again periodic boundary conditions, we set $\V = \Hper$ and $\cH= L^2(\Omega)$ with the standard inner products. In this case, the sesquilinear form $a$ has the form
\[
	a(u, v) 
	:= \int_{\Omega}\frac{1}{\eps(x)} \nabla_\k u\cdot \overline{\nabla_\k v} \dx.
\]
The proof of the G{\aa}rding inequality follows similarly as in Lemma~\ref{lem:betaTM}. 
\end{example}
More general eigenvalue problems are discussed in Sections~\ref{sect:nonlinearHermit} and~\ref{sect:nonlinear}. There we consider a TM mode with an electric permittivity, which depends on the frequency $\omega$. This then leads to a nonlinear eigenvalue problem.
%
%
\subsection{Operator formulation}\label{sect:formulation:op}
The weak formulation of the eigenvalue problem~\eqref{eq:evp:weak} can be equivalently written as an operator equation in the (conjugate) dual space of $\V$, i.e., the space of conjugate linear and continuous mappings from $\V \to \C$. This then yields a convenient formulation for the introduction of iterative schemes. We introduce the operator $\calA\colon \V\to \V^*$ by 
\[
  \langle \calA u, v \rangle := a(u,v).
\]
Further, we define $\calI\colon \V\to \V^*$ as the embedding of $\V$ in $\V^*$ induced by the Gelfand triple $\V$, $\cH$, $\V^*$ from Assumption~\ref{ass_spaces} with respect to the inner product in $\cH$, cf.~\cite[Ch.~23.4]{Zei90a}. To be precise, this means~$\langle \calI u, v \rangle = (u,v)$ for all $u, v\in \V$. The operator equation corresponding to~\eqref{eq:evp:weak} then reads
\begin{align}
\label{eq:evp:op}
	\calA u = \lambda\, \calI u
	\qquad\text{in } \V^*. 
\end{align}
Note that this equation is stated in the dual space of $\V$, which means that we consider test functions in $\V$ as in~\eqref{eq:evp:weak}. Hence, the two formulations \eqref{eq:evp:weak} and \eqref{eq:evp:op} are equivalent. 

Finally, recall the definition of the shifted sesquilinear form $\ab$ from Remark~\ref{rem:shift}. With this, we define the corresponding operator $\calAb:= \calA+\beta\calI \colon \V\to \V^*$, which is then positive and thus, invertible.   	
\section{Iterative Methods on Operator Level}\label{sect:convHerm}
In this section, we analyze iterative methods to find the smallest eigenvalue as well as the corresponding eigenfunction of the operator eigenvalue problem~\eqref{eq:evp:op}. We emphasize that we do not apply any spatial discretization but perform the eigenvalue iteration directly to the operator equation. 

We first consider the inverse power method for which we prove the convergence in $\V$ under the assumptions collected in the previous section. For this we consider two variants of the method and discuss the commutativity of spatial discretization and eigenvalue iteration, assuming an appropriate normalization within the scheme. 
Second, we discuss Arnoldi's method for the operator case, based on Krylov subspaces as in the matrix case. 
%
%
\subsection{Inverse power method}\label{sect:convHerm:power}
Power and inverse power methods come in various variants with different kinds of scaling, see e.g.~\cite[Ch.~10.3]{AllK08} or \cite[Ch.~4]{Saa11} for the matrix case. One may even consider the scaling with the exact eigenvalue $\lambda$ as done, e.g., in~\cite{EasE07, AltHP18ppt}. Clearly, the latter is only of interest for theoretical observations rather than actual computations. 

As for every iteration scheme we assume a starting function $u^0\in\V$. In order to permit the iterates to converge to the wanted eigenfunction, one needs an additional assumption on~$u^0$, e.g., having a non-vanishing component in the direction of this eigenfunction. 
%
%
\subsubsection{Rayleigh quotient iteration}\label{sect:convHerm:Rayleigh}
A direct application of the inverse power method for the operator case would apply the inverse of the differential operator over and over again. However, this operator may not be invertible and the {\em Rayleigh quotient}
$
\langle \calA u^j, u^j\rangle/\langle \calI u^j, u^j \rangle
$ is not guaranteed to remain positive. Thus, we consider the shifted eigenvalue problem 
\begin{align}
\label{eq:evp:op:shift}
	\calAb u 
	:= (\calA + \beta\calI)\, u 
	= (\lambda + \beta)\, \calI u
	=: \mu\, \calI u
	\qquad\text{in } \V^*. 
\end{align}
This then leads to the following algorithm: 
Given an initial function $u^0 \in \V$, $u^0 \neq 0$, we solve for $j=1,2,\dots$ the variational problem  
\begin{align}
\label{eq:iter:power}
	\calAb u^j
	= \mu^{j-1}\, \calI \tilde u^{j-1}
	\qquad\text{in } \V^*.
\end{align}
Therein, $\tilde u^{j} := u^{j} / \Vert u^{j}\Vert$ includes the normalization in $\cH$, which is the natural norm corresponding to the eigenvalue problem~\eqref{eq:evp:op}. 
Further, $\mu^j$ denotes the Rayleigh quotient, i.e., 
\[
  \mu^j 
  := \frac{\langle \calAb u^j, u^j\rangle}{\langle \calI u^j, u^j \rangle}
  = \frac{\ab(u^j, u^j)}{(u^j, u^j)}
  = \frac{\Vert u^j\Vert^2_\beta}{\Vert u^j\Vert^2}
  > 0.
\]
We emphasize that the iterates $u^j$ of the power iteration~\eqref{eq:iter:power} are not normalized. 	
The presence of the Rayleigh quotient within the iteration directly provides an approximation of the eigenvalue $\lambda$, given by $\lambda^j := \mu^j - \beta$. It remains to discuss the convergence of the suggested iteration. 
\begin{theorem}
\label{thm:power}
Given Assumptions~\ref{ass_spaces}-\ref{ass_a}, the power iteration~\eqref{eq:iter:power} converges to an eigenpair $(\ulim, \llim)$ of~\eqref{eq:evp:op} in the sense that the (sub)sequences $u^j$ and $\lambda^j := \mu^j-\beta$ satisfy $u^j \conv \ulim$ in~$\V$ and $\lambda^j \conv \llim$ with $\calA \ulim = \llim\calI \ulim$ in $\V^*$.
\end{theorem}
\begin{proof}
We follow the proof in~\cite{Boz16} where the convergence of the inverse power method for the $p$-Laplacian is shown. 
The first step is to show the monotonic decrease of the sequence $\mu^j$. For this, we consider~\eqref{eq:iter:power} with test functions $u^j$ and $\tilde u^{j-1}$ leading to 
\[
	\Vert u^j\Vert^2_\beta 
  	\le \mu^{j-1} \Vert u^j \Vert, \qquad
	\mu^{j-1}
    = \mu^{j-1} \Vert \tilde u^{j-1}\Vert^2 
    \le \sqrt{\mu^{j-1}} \Vert u^j\Vert_\beta.  	
\]
A combination of these two estimates yields $\Vert u^j \Vert_\beta \le \sqrt{\mu^{j-1}} \Vert u^j \Vert$ and thus,
\begin{align}
\label{eqn_proof_a}
  \sqrt{\mu^j}
  = \frac{\Vert u^j\Vert_\beta}{\Vert u^j\Vert}
  \le \sqrt{\mu^{j-1}}.
\end{align}
The monotonicity and $\mu^j > 0$, which follows from the positivity of $\calAb$, imply that there exists the limit $\mulim := \lim_{j \to \infty} \mu^j \ge 0$. 
It even holds $\mulim > 0$, since the $\mu^j$ are uniformly bounded away from zero. 

As a second step, we conclude from estimate~\eqref{eqn_proof_a} that $\Vert \tilde u^j\Vert_\beta = \sqrt{\mu^j} \le \sqrt{\mu^0}$. Since the norms $\Vert\cdot\Vert_\beta$ and $\Vert \cdot\Vert_\V$ are equivalent, we obtain that $\tilde u^j$ is uniformly bounded in~$\V$. Thus, there exists a convergent subsequence and an element $\ulimt \in \V$, which satisfy (without relabeling) 
\[
  \tilde u^j \wconv \ulimt \quad\text{in }\V, \qquad
  \tilde u^j \conv \ulimt \quad\text{in }\cH. 
\]
Note that we have used here the compact embedding $\V \hook \cH$ from Assumption~\ref{ass_compact}. Obviously, we have $\Vert \ulimt \Vert = \lim_{j\to \infty} \Vert \tilde u^j\Vert = 1$. 
Further, we know from previous calculations that 
\[
  \Vert u^j \Vert_\beta^2
  \le \mu^{j-1} \Vert u^j\Vert
  \le \mu^0 \Vert u^j\Vert
  \le C_{\V\hook\cH} \mu^0 \Vert u^j\Vert_\V
  \le C_{\V\hook\cH} \mu^0 \alpha^{-\frac{1}{2}} \Vert u^j \Vert_\beta. 
\]
This means that the sequence $u^j$ is uniformly bounded in $\V$. Thus, there exists a limit $\ulim \in\V$ such that (again without relabeling of the subsequence) 
\[
  u^j \wconv \ulim \quad\text{in }\V, \qquad
  u^j \conv \ulim \quad\text{in }\cH. 
\]
%
In the following we compare the two limits $\ulim$ and $\ulimt$. For this, we consider 
\[
  \mu^j
  = \frac{a_\beta( u^j, u^j)}{\Vert u^j\Vert^2}  
  = \mu^{j-1} \frac{(\tilde u^{j-1}, u^j)}{\Vert u^j\Vert^2}  
  = \frac{\mu^{j-1}}{\Vert u^j\Vert} (\tilde u^{j-1},\tilde  u^j)
\]
Taking the limit $j\to\infty$ on both sides, we conclude that 
\[
  \mulim
  = \frac{\mulim}{\Vert \ulim \Vert} \Vert \ulimt \Vert^2 
  = \frac{\mulim}{\Vert \ulim \Vert},  
\]
i.e., $\Vert \ulim \Vert = 1$. As a result, the sequence $(\tilde u^j)$ converges to $\ulim/\Vert \ulim\Vert = \ulim$. The uniqueness of the limit then yields $\ulim = \ulimt$. 

To show that the pair $(\ulim, \llim)$ with $\llim := \mulim-\beta$ is an eigenpair of~\eqref{eq:evp:op}, we apply the limit to equation~\eqref{eq:iter:power}. 
We conclude that $(\ulim, \mulim)$ indeed solves 
\begin{align}
\label{eqn:reviewer2}
  a_\beta(\ulim, v) = \mulim (\ulim, v)\qquad \text{for all } v\in \V.
\end{align}
Thus, we have $a(\ulim, v) = \llim (\ulim, v)$ for all $v\in\V$ or, in operator form, $\calA \ulim = \llim \calI \ulim$. The weak convergence $\tilde u^j \wconv \ulimt = \ulim$ in $\V$ additionally implies, using~\eqref{eqn:reviewer2} with $v = u^\star$, 
\[
  \sqrt{\mulim} 
  = \Vert \ulim \Vert_\beta 
  \le \liminf_{j\to \infty} \Vert \tilde u^j \Vert_\beta 
  = \liminf_{j\to \infty} \sqrt{\mu^j}
  = \sqrt{\mulim}. 
\]
Note that the inequality is strict, if and only if the convergence is not strong. This implies $\tilde u^j, u^j \conv \ulim$ in $\V$. If~$\mulim$ is a single eigenvalue, then every convergent subsequence has the same limit.  
\end{proof}
\begin{remark}
In order to show that $(\ulim, \llim)$ is the smallest eigenpair, one needs an additional assumption on $u^0$. In the real case, one may consider $u^0\ge 0$, using the maximum principle, see e.g.~\cite{Boz16}. 
\end{remark}
\begin{remark}
A second strategy to prove Theorem~\ref{thm:power} is to reformulate the eigenvalue problem in terms of the resolvent. Assumption~\ref{ass_compact} shows that the resolvent is a compact operator such that the results in~\cite{EriSL95} can be applied.
\end{remark}
%
\subsubsection{An alternative power method}
In Theorem~\ref{thm:power} we have shown that the inverse power method~\eqref{eq:iter:power} provides in the limit an eigenvalue (the limit of the Rayleigh quotient) and an eigenfunction. However, one may also omit the Rayleigh quotient, which leads to the iteration 
\begin{align}
\label{eq:iter:power2}
  \calAb v^j
  = \calI \tilde v^{j-1}
  \qquad\text{in } \V^*
\end{align}
and the following convergence result. 
\begin{lemma}
\label{lem:power2}
Assume $u^0=v^0 \in \V$ with $\Vert u^0\Vert = 1$. Let $u^j$ and $v^j$ be the sequences obtained from the iteration procedures~\eqref{eq:iter:power} and~\eqref{eq:iter:power2}, respectively. Then, we have the relation $u^j = \mu^{j-1} v^j$ for all $j\ge 1$ with the Rayleigh quotient $\mu^j = \Vert u^j\Vert_\beta^2 / \Vert u^j\Vert^2$. 
\end{lemma}
\begin{proof}
We prove this result by mathematical induction and observe first that
\[
	u^1
	= \mu^0 \calAb^{-1} \calI u^0
	= \mu^0 \calAb^{-1} \calI v^0
	= \mu^0 v^1.
\]
Now, assuming that $u^j = \mu^{j-1} v^j$ is true for a fixed but arbitrary index~$j$, we obtain
\[
	u^{j+1}
	= \mu^{j} \calAb^{-1} \calI \frac{u^j}{\Vert u^j\Vert}
	= \mu^{j} \calAb^{-1} \calI \frac{\mu^{j-1} v^j}{\mu^{j-1} \Vert v^j\Vert}
	= \mu^{j} \calAb^{-1} \calI \frac{v^j}{\Vert v^j\Vert}
	= \mu^{j} v^{j+1}. 
\]
Note that we have used the fact that $\mu^j>0$ and thus $|\mu^j| = \mu^j$. 
\end{proof}
Lemma~\ref{lem:power2} directly implies the convergence of the iteration~\eqref{eq:iter:power2}. In particular, we have $v^j = u^j / \mu^{j-1} \to \ulim/\mulim =: \vlim$ in $\V$. Thus, the limit is not normalized but rather satisfies 
\[
  \Vert \vlim\Vert
  = \frac{\Vert \ulim \Vert}{|\mulim|} 
  = \frac{1}{\mulim}.
\] 
%
\subsubsection{Commutativity of discretization and power iteration}
The convergence of the power method in the operator case directly leads to the question whether the application of the power method and the spatial discretization commute. 
%
If we discretize the shifted eigenvalue problem~\eqref{eq:evp:op:shift} by finite elements, then we obtain a system of the form 
\[
  Kq = \mu\, M q.
\]
Therein, $q \in \C^n$ encodes a finite-dimensional approximation of the eigenfunction $u \in \V$, e.g., the coefficients w.r.t.~a finite element basis. Since we have included the boundary conditions in the space $\V$, we assume that the stiffness matrix $K \in \C^{n,n}$ and the mass matrix $M \in \C^{n,n}$ are Hermitian and positive definite. With $A := M^{-1}K$ the discrete system is equivalent to the eigenvalue problem $Aq = \mu q$. 
Seeking for the smallest eigenvalue, we apply the {\em inverse power method} with normalization, i.e., 
\begin{align}
\label{eq:firstLinearize}
  q^j 
  = A^{-1} \tilde q^{j-1} 
  = A^{-1} \frac{q^{j-1}}{\Vert q^{j-1}\Vert_Z} 
  = K^{-1}M \frac{q^{j-1}}{\Vert q^{j-1}\Vert_Z},
\end{align}
where $Z$ is a symmetric and positive definite matrix and $\Vert q\Vert_Z := (q^T Zq)^{1/2}$ the corresponding norm. As approximation for the eigenvalue we consider $\mu^j := 1 / \Vert q^j\Vert_Z$. The starting vector is denoted by $q^0$. We emphasize that the iteration converges despite of the choice of the norm as long as $q^0$ contains a non-zero component in direction of the first eigenvector.

%
We now consider the spatial discretization of~\eqref{eq:iter:power2}, i.e., we first apply the inverse power method to the PDE eigenvalue problem~\eqref{eq:evp:op:shift} and then discretize. We emphasize that this allows a different discretization scheme in each iteration step and thus, adaptivity. If we consider, however, the same spatial mesh as before in each iteration step, then the same matrices $K$ and $M$ appear and we get 
\[
  K q^j 
  = M \tilde q^{j-1}
  = M \frac{q^{j-1}}{\Vert q^{j-1} \Vert_M}.
\]
Note that the applied normalization is here w.r.t the $M$-norm, since this corresponds to the $L^2$-inner product in the infinite-dimensional case. Thus, the resulting iteration equals~\eqref{eq:firstLinearize} if we choose the normalization matrix $Z=M$. 
%
%
\subsection{A Krylov subspace method}
The natural extension of the inverse power method in order to approximate several eigenvalues is a subspace iteration, cf.~\cite[Ch.~5]{Saa11}. This includes several starting functions, for which a power iteration is applied, and an additional orthogonalization step is performed. The computation of several eigenvalues is also of interest in the calculation of band-gaps of a photonic crystal.
Here, we consider the Arnoldi method in the operator setting. For this, we need an extension of the Krylov subspaces used in numerical linear algebra. 
%
\subsubsection{Krylov spaces}
Krylov subspaces play a crucial role for iterative eigenvalue computations, see, e.g., \cite[Ch.~6.1]{Saa11}. In order to generalize these methods to the PDE setting, we need Krylov subspaces for general Hilbert spaces~\cite{GueHS14}.

Let $u^0$ be a function in $\V$, e.g., an initial guess for the power method in Section~\ref{sect:convHerm:power}. With this, we define the {\em Krylov subspace}  
\begin{align}
\label{def_krylov}  
  \krylov_\beta^m(u^0) 
  := \sspan \big\{ u^0,\ \calAb^{-1}\calI u^0,\ \dots,\ (\calAb^{-1}\calI)^{m-1} u^0  \big\} \subseteq \V.
\end{align}
Obviously, this defines a closed subspace of $\V$. We emphasize that $\krylov_\beta^m(u^0)$ is spanned -- as in the finite-dimensional case -- by the iterates of the power method. To see this, note that $(\calAb^{-1}\calI)^{j} u^0$ equals the corresponding iterate of the power method up to a multiplicative constant. Thus, if we denote the sequence resulting from the modified inverse power method~\eqref{eq:iter:power2} by~$u^j_\text{pow}$, then we have
\[
  \krylov_\beta^m(u^0) 
  = \sspan \big\{ u^0,\ u^1_\text{pow},\ \dots,\ u^{m-1}_\text{pow} \big\}. 
\]
%
%
\subsubsection{Arnoldi's method}
We translate the Arnoldi algorithm from the matrix setting~\cite[Ch.~6.2]{Saa11} to the present operator case. 
A similar infinite-dimensional Arnoldi algorithm is considered in~\cite{JarMM12}. 
Let $\{v_1, \dots, v_m\}$ denote a basis of $\krylov_\beta^m(u^0)$, e.g., obtained by a Gram-Schmidt orthogonalization process. Then, the new iterate of the Arnoldi method is given by $u^m := \sum_{j=1}^m \zeta_j v_j \in \krylov_\beta^m(u^0)$, whose coefficients $\zeta := [\zeta_1, \dots, \zeta_m]^T$ and corresponding $\mu^m\in\R$ are derived by the Galerkin projection
\[
  \sum_{j=1}^m \zeta_j\, \ab(v_j,v_k) 
  = \mu^m \sum_{j=1}^m \zeta_j\, (v_j, v_k) 
  \qquad\text{for } k=1,\dots, m. 
\]
This is equivalent to the $m\times m$ eigenvalue problem $\tilde K \zeta = \mu^m \tilde M \zeta$, for which we search for the smallest eigenpair. Here, $\tilde K$ and $\tilde M$ are stiffness and mass matrices restricted to the Krylov basis, i.e., 
\[
  \tilde K_{kj} := \ab(v_j,v_k), \qquad
  \tilde M_{kj} := (v_j,v_k).
\]
Thus, the extra costs going from the power method to the Arnoldi method are identical as in the finite-dimensional setting, namely the solution of a small (but dense) eigenvalue problem. Note that the resulting approximation of the eigenvalue, namely~$\mu^m$, is again the Rayleigh quotient of the iterate $u^m$.
\begin{remark}
The commutativity result for the inverse power method does not carry over to the Arnoldi method, since the discretization of the operator $(\calAb^{-1}\calI)^{j}$ leads, in general, not to the matrix $(K^{-1}M)^{j}$. 
\end{remark}
The obtained pair of the Arnoldi method provides the best-approximation within the Krylov subspace $\krylov_\beta^m(u^0)$ in the sense that 
\[
  \langle \Res(u^m, \mu^m), v \rangle = 0 
\]
for all $v \in \krylov_\beta^m(u^0)$ and with the residual defined by $\Res(u, \mu) \coloneqq \calAb u -\mu \calI u \in\V^*$ for any pair $(u, \mu)\in \V\times \C$. Obviously, this implies that the Arnoldi method is superior to the inverse power method. The norm of the residual may also be used as an error estimator as it can be transformed to the backward error. Thus, small residuals indicate good approximations of the eigenpair~\cite[Ch.~4]{Mie11}.

The gain of the Arnoldi method can also be characterized in terms of the Courant {\em min-max principle} in Hilbert spaces, cf.~\cite{Cou20} or~\cite[Ch.~1]{WeiS72}. 
This means that the $\ell$th eigenvalue is defined by minimizing over all $\ell$-dimensional subspaces of $\V$, i.e., 
\[
  \lambda_\ell + \beta
  = \mu_\ell 
  = \min_{\substack{ \V^{(\ell)} \subset\, \V,\\ \dim \V^{(\ell)} =\, \ell}} \max_{\ v\, \in\, \V^{(\ell)}} \frac{a_\beta(v,v)}{(v,v)}.
\]
With this, one shows that computed approximations of the eigenvalues are larger than the exact ones. With the same arguments one can show that the Arnoldi iteration provides better approximations than the inverse power method and thus, converges as well. More precisely, the Arnoldi method computes an approximation $\mu^m_\text{Arnoldi}$ satisfying 
\[
  \mu_1
  \le \mu^m_\text{Arnoldi} 
  = \min_{\substack{\V^{(1)} \subset\, \krylov_\beta^m(u^0),\\ \dim \V^{(1)} = 1}} \max_{\ v\in \V^{(1)}} \frac{a_\beta(v,v)}{(v,v)}
  \le \max_{\ v\, \in\, \sspan\{ u^{m-1}_\text{pow} \} } \frac{a_\beta(v,v)}{(v,v)} 
  = \mu^{m-1}_\text{pow} .
\]
Note that the inequality holds, since $u^{m-1}_\text{pow}$ is an element of the Krylov subspace $\krylov_\beta^m(u^0)$ and thus, $\sspan\{ u^{m-1}_\text{pow} \}$ is a particular one-dimensional subspace. 
\section{The Inverse Power Method for a Nonlinear Model Problem}\label{sect:nonlinearHermit}
In this section, we consider an extension of the TM mode from Example~\ref{exp:TM}, in which the relative electric permittivity~$\eps$ depends on the frequency and thus, the eigenvalue. This then leads to a nonlinear eigenvalue problem. Assuming that $\eps$ is a rational function in the frequency, we are able to reformulate the eigenvalue problem to a linear one satisfying Assumptions~\ref{ass_spaces} and~\ref{ass_a}. 
This linearization, however, leads to a lack of compactness, which calls for a novel convergence analysis of the inverse power method. 
This then leads to a slightly weaker convergence result, compared to Section~\ref{sect:convHerm}. 
%
%
\subsection{A simplified Drude-Lorentz model}
We consider a photonic crystal made up of two different materials. For this, we decompose the computational domain $\Omega$ into two subdomains $\Omega=\Omega_1\cup\Omega_2$, each representing one material. We define the corresponding indicator functions on $\Omega_j$ by $\chi_j\colon \Omega\to \{1,0\}$, $j=1,2$. On both subdomains we assume the relative electric permittivity to be constant in space and thus, 
\begin{equation}
\label{eq:distributionPermittivity}
	\eps(x,\omega) = \eps_1(\omega)\chi_1(x)+\eps_2(\omega)\chi_2(x).
\end{equation}
For the sake of brevity, the material contained in $\Omega_1$ is assumed to be linear, i.e., we set the relative permittivity in this subdomain to a constant $\eps_1(\omega) \equiv \alpha_1>0$. For the frequency dependence in the second material we consider a simplified version of the Drude-Lorentz model, see e.g.~\cite{LuoL10} or \cite[Ch.~7.5]{Jac99}. 
More precisely, we assume the electric permittivity to be of the form 
\begin{equation}
\label{eq:DrLoModel}
	\eps_2(\omega) =
	\alpha_2 + \sum\limits_{\ell=1}^L \frac{\xi_\ell^2}{\eta_\ell^2-\omega^2}
\end{equation}
with a positive constant $\alpha_2 > 0$ and real parameters $\eta_\ell, \xi_\ell$ such that $\eta_\ell^2, \xi_\ell^2 \ge 0$. This corresponds to the 'lossless' case considered in~\cite{Eng10}. In order to stay bounded, we only consider~$\omega$ away from the poles given by~$\eta_\ell$. Since $\omega$ appears in~\eqref{eq:DrLoModel} only squared, we introduce $\lambda\coloneqq \omega^2$ and obtain 
\[ 
	\lambda\, \eps_2(\lambda)
	= \lambda\, \alpha_2 + \sum_{\ell=1}^L \frac{\lambda\,  \xi_\ell^2}{\eta_\ell^2-\lambda}	
	= \lambda\, \alpha_2 - \Xi + \sum\limits_{\ell=1}^L \frac{\xi_\ell^2\eta_\ell^2}{\eta_\ell^2-\lambda} 
\]
with 
\[
	\Xi 
	:= \sum_{\ell=1}^L \xi_\ell^2
	\ge 0.
\]
Due to the inclusion of $\Xi$, the fractional terms are in a strictly proper form, i.e., the degree of the polynomial in terms of $\lambda$ in the numerator is strictly smaller than the degree in the denominator. All in all, this leads to the nonlinear eigenvalue problem 
\begin{align}
\label{eq:nonlinEVP:realCase}
	-  \nabla_\k \cdot \nabla_\k u_\k(x) &+ \Xi\, \chi_2(x)u_\k(x) \notag \\
	&\quad = \lambda\, \big( \alpha_1\chi_1(x)+\alpha_2\chi_2(x) \big)u_\k(x) +  \sum\limits_{\ell=1}^L \frac{\xi_\ell^2\eta_\ell^2}{\eta_\ell^2-\lambda} \chi_2(x)u_\k(x)  
\end{align}
with $\nabla_\k$ denoting again the shifted gradient introduced in Example~\ref{exp:TM}. Our aim is to turn this into a linear eigenvalue problem in order to apply the iterative methods from the previous section. For this, we follow the ideas presented in~\cite{SuB11, EffKE12, Eff13}, which consider the corresponding finite-dimensional case. The main clue is to rewrite the sum, which may be regarded as a transfer function, by means of a realization, i.e., 
\begin{equation}
\label{eq:realization:realCase}
	\sum\limits_{\ell=1}^L \frac{\xi_\ell^2\eta_\ell^2}{\eta_\ell^2-\lambda}
	= \bVec^\ast\big(\AMat-\lambda\IMat \big)^{-1}\bVec.
\end{equation}
Due to the simple structure of the permittivity, we can directly read off the vector  $\bVec=[\xi_1\eta_1, \ldots, \xi_L\eta_L]^T \in \mathbb{\R}^L$ and $\AMat\in \R^{L\times L}$ as the diagonal (and thus Hermitian) matrix with $\AMat_{j,j}=\eta_j^2$ for $j=1,\ldots,L$. By $\IMat\in \R^{L\times L}$ we denote the identity matrix. 
\begin{remark}
For other models of the permittivity, the choice (and even the dimension) of $\AMat$ and $\bVec$ may not be as straightforward. Further, the identity matrix $\IMat$ may be replaced by another positive Hermitian matrix. Proper realizations for such cases may be found using the techniques discussed in~\cite{SchUBG18}.
\end{remark}
%
%
\subsection{Spaces and embeddings}
For the weak formulation and linearization of the eigenvalue problem~\eqref{eq:nonlinEVP:realCase} we introduce several function spaces. First, we introduce 
\[
  H\coloneqq L^2(\Omega), \qquad
  V\coloneqq \Hper, \qquad 
  \LOmgII\coloneqq\{ v\in H\ |\ v \text{ vanishes on } \Omega_1 \}.
\] 
These spaces form Hilbert spaces equipped with the inner products 
\begin{align*}
  (u, v)\coloneqq(u,v)_H\coloneqq\int_{\Omega} u\overline{v} \dx,\quad
  (u, v)_V\coloneqq (u,v)+(\nabla u, \nabla v),\quad
  (u, v)_\LOmgII\coloneqq \int_{\Omega_2} u\overline{v} \dx 
\end{align*}
for $u$ and $v$ in the respective spaces $H$, $V$, or $\LOmgII$. Second, we define the product spaces 
\[
  \cH\coloneqq H \otimes \LOmgII^L, \qquad 
  \V\coloneqq V \otimes \LOmgII^L
\]
with $\LOmgII^L := \LOmgII \otimes \dots \otimes \LOmgII$. 
Also these product spaces are Hilbert spaces. In $\cH$ we consider the inner product 
\[
  (\zVec_1, \zVec_2)_\cH
  \coloneqq (u, v) + (\xVec, \yVec)_{\LOmgII^L} 
  \coloneqq (u, v) + \sum_{\ell=1}^L\, (x_{\ell}, y_{\ell})_{\LOmgII}
\]
for $\zVec_1 = [u; \xVec]$, $\zVec_2 = [v; \yVec] \in\cH$ consisting of $u,v\in H$ and $\xVec, \yVec \in \LOmgII^L$ with $\xVec = [x_1; \dots; x_L]$, $\yVec = [y_1; \dots; y_L]$. 
Note that we use here the notation $[u; \xVec] := [u, \xVec^T]^T$. Analogously, we define an inner product in~$\V$ by replacing $(u, v)$ by $(u, v)_V$. We also define the corresponding norms $\Vert \zVec \Vert_\cH^2 := (\zVec, \zVec)_\cH$ and $\Vert \zVec \Vert_\V^2 := (\zVec, \zVec)_\V$.
\begin{remark}
\label{rem_notComp}
Although the embedding $V \hook H$ is compact, the embedding $\V \hook \cH$ is not. This is due to the fact that the identity operator is not compact in infinite dimensions.  
\end{remark}
For the weak formulation we further need several embeddings. First, $\calI\colon V\to V^*$ denotes the continuous inclusion map defined by the Gelfand triple $V$, $H$, $V^\ast$, cf.~Section~\ref{sect:formulation:op}. Second, we define the extension of the mapping $u \mapsto (u, \cdot\, )_\LOmgII$ as $\calI_2\colon V\to V^*$, i.e.,  
\[
  u 
  \mapsto \langle \calI_2 u, \cdot\, \rangle_{V^\ast, V}
  := (\chi_2 u,\,  \chi_2\, \cdot\, )_\LOmgII
  = \int_{\Omega_2} u\, \overline{\, \cdot\,  } \dx.
\]
In the same manner, but based on the indicator function $\chi_1$, we define $\calI_1\colon V\to V^*$. The weighted combination of these two embeddings yields $\calIa\colon V\to V^*$, given by 
\[
  \calIa\coloneqq \alpha_1 \calI_1 + \alpha_2 \calI_2.
\]
Finally, we introduce the embedding~$\oI\colon \LOmgII\to V^\ast$. For $u\in \LOmgII$ this is defined by 
\[
  u 
  \mapsto \langle \oI u, \cdot\, \rangle_{V^\ast, V}
  := (u,\,  \chi_2\, \cdot\, )_\LOmgII
  = \int_{\Omega_2} u\, \overline{\, \cdot\,  } \dx .
\]
The corresponding dual operator $\oI^*\colon V \to \LOmgII^*$ satisfies for $v\in V$ and $u \in \LOmgII$ that 
\[
  \langle \oI^* v, u\rangle_{\LOmgII^*, \LOmgII}
  = \langle v, \oI u\rangle_{V, V^*}
  = (\chi_2 v, u)_\LOmgII.
\]
\begin{lemma}
\label{lem_I2I2}
With the Riesz isomorphism $j_\LOmgII\colon \LOmgII\to \LOmgII^\ast$, the introduced embeddings satisfy the identity $\oI\, j_\LOmgII^{-1}\, \oI^\ast=\calI_2\colon V\to V^*$. 
\end{lemma}
\begin{proof}
Consider $u,v\in V$. Then, the claimed identity can be seen by 	
\[
	\langle \oI\, j_\LOmgII^{-1}\, \oI^\ast u, v \rangle_{V^\ast, V}
	= \langle \oI^\ast u, \chi_2 v \rangle_{X^\ast, X}
	= (\chi_2 u, \chi_2 v)_\LOmgII= \langle \calI_2 u, v \rangle_{V^\ast, V}. \qedhere
\]
\end{proof}
\begin{remark}
In the remainder of this paper, we often omit to write the Riesz isomorphisms $j_H\colon H\to H^\ast$ or $j_\LOmgII\colon \LOmgII\to \LOmgII^\ast$ if their presence is clear from the context. Thus, we may write $\oI\oI^\ast = \calI_2$.
\end{remark}
%
%
\subsection{Weak formulation and linearization}
This subsection is devoted to the transformation of the nonlinear eigenvalue problem into a linear one by the introduction of new variables. Thus, we aim to write~\eqref{eq:nonlinEVP:realCase} in the form    
\[
	\bigAMat \zVec = \lambda\, \bigIMat \zVec.
\]
Based on the proper form of the permittivity given in~\eqref{eq:realization:realCase}, we obtain the following weak form of the eigenvalue problem. For a given (and fixed) wave vector~$\k \in \calK$, find a non-trivial pair~$(u_\k, \lambda) \in V\times \R$ such that
\begin{equation}
\label{eq:nonlinearWeak:realCase}
	\Ak u_\k + \Xi\,  \calI_2 u_\k 
	= \lambda\, \calI_{\alpha} u_\k + \bVec^\ast\big(\AMat-\lambda\IMat \big)^{-1}\bVec\,  \calI_2 u_\k.
\end{equation}
We emphasize that this equation is stated in $V^*$. The operator~$\Ak\colon V\to V^*$ denotes the weak form of the shifted Laplacian, cf.~Example~\ref{exp:TM}. For the first order formulation of \eqref{eq:nonlinearWeak:realCase} we introduce a new variable 
\begin{align}
\label{def:xy:realCase}
  \xVec 
  \coloneqq  (\AMat - \lambda\IMat)^{-1}\bVec\, \oI^\ast u_\k \in \LOmgII^{L}.
\end{align}
Note that this includes a hidden application of $j_\LOmgII^{-1}$. With Lemma~\ref{lem_I2I2} this leads to a linear eigenvalue problem where we search for a pair $(\zVec, \lambda)$ with $\zVec := [u_\k; \xVec]\in \calV$ such that 
\begin{equation}
\label{eq:linearizedEVP:realCase}
	\left[\begin{array}{ccc}
	\Ak + \Xi\, \calI_2 & - \oI\bVec^* \\
	- \bVec\, \oI^* & \AMat\, j_\LOmgII  \end{array}\right] \zVec
	= \lambda \left[\begin{array}{ccc} \calIa & \\
		& \IMat\, j_\LOmgII \end{array}\right] \zVec.
\end{equation}
This formulation consists of two equations stated in the dual spaces of~$V$ and~$\LOmgII^L$, respectively. The Riesz isomorphism~$j_\LOmgII$ should be understood as a componentwise application.
\begin{remark}
The operator on the left-hand side of the first order eigenvalue problem~\eqref{eq:linearizedEVP:realCase} has a generalized saddle point structure. 
\end{remark}	
\begin{lemma}
The eigenvalue problems~\eqref{eq:nonlinearWeak:realCase} and~\eqref{eq:linearizedEVP:realCase} are equivalent. This means that an eigenpair $(u_\k, \lambda)$ of~\eqref{eq:nonlinearWeak:realCase} defines a solution of \eqref{eq:linearizedEVP:realCase} by $([u_\k; \xVec], \lambda)$ with $\xVec$ defined as in~\eqref{def:xy:realCase} and vice versa. 
\end{lemma}
\begin{proof}
The second block row of~\eqref{eq:linearizedEVP:realCase} is given by $(\AMat-\lambda\IMat)j_\LOmgII \xVec = \bVec\, \oI^\ast u_\k$, which implies \eqref{def:xy:realCase}. Substituting this formula for $\xVec$ into the first block row yields together with $\oI\oI^\ast=\calI_2$ from Lemma~\ref{lem_I2I2} that \eqref{eq:linearizedEVP:realCase} is indeed equivalent to the nonlinear eigenvalue problem~\eqref{eq:nonlinearWeak:realCase}. 
\end{proof}
Defining $\bigAMat, \bigIMat\colon \V\to\V^*$ in an obvious manner, we can write~\eqref{eq:linearizedEVP:realCase} in the form $\bigAMat \zVec = \lambda\, \bigIMat \zVec$. The sesquilinear form corresponding to~$\bigIMat$ reads $\bigIForm\colon \calV\times \calV\to \C$, 
\begin{align*}
	\bigIForm(\zVec_1, \zVec_2)
	:= \big\langle \calIa u, v \big\rangle_{V^*,V} 
	+ \big(\xVec_1, \xVec_2 \big)_{\LOmgII^L}
\end{align*}
for $\zVec_1 = [u; \xVec_1]$, $\zVec_2 = [v; \xVec_2] \in \calV$. Note that we may also consider~$\bigIForm$ as a sesquilinear form mapping from $\calH\times \calH$ to $\C$. We now show that $\bigIForm$ defines an inner product in $\calH$. 
\begin{lemma}
\label{lem:bigIForm:first}	
The sesquilinear form $\bigIForm\colon \calH\times\calH\to \C$, defines an inner product in $\calH$. 
\end{lemma}
\begin{proof}
Obviously, $\bigIForm$ is Hermitian and sesquilinear. Further, for any $\zVec=[u; \xVec] \in \calH$ and $\aalpha := \min\{\alpha_1, \alpha_2\}>0$, it holds that
\[
	\bigIForm(\zVec,\zVec)
	\geq \aalpha\, \Vert u\Vert^2 
	+ \big(\xVec, \xVec \big)_{\LOmgII^L}
	\geq \min(\aalpha, 1)\, \Vert \zVec\Vert_\calH^2, 
\]
which proves its positivity.
\end{proof}
%
%
\subsection{Shifted eigenvalue problem}
In the previous subsection, the nonlinear eigenvalue problem was brought into the form $\bigAMat \zVec = \lambda\, \bigIMat \zVec$. In order to use the framework presented in Section~\ref{sect:formulation}, we need to apply a shift to gain positivity of the differential operator. Recall from the discussion of the linear case that the operator $\Ak$ is not elliptic. From Lemma~\ref{lem:betaTM} we know, however, that $\Aktb := \Ak + \tbeta\calI$ is elliptic for every $\tbeta>0$. 

For fixed $\tbeta>0$ and $\aalpha := \min\{\alpha_1, \alpha_2\}$ we introduce $\beta := \tbeta/\aalpha$ and shift the linearized eigenvalue problem by $\beta\bigIMat\zVec$. This then provides the appearance of $\Aktb$. More precisely, we obtain the shifted operator
\begin{equation}
\label{eq:linearizedEVP:realCase:shift}
  \bigAMat_\beta
  := \bigAMat + \beta\, \bigIMat
  = \left[\begin{array}{cc}
  \Ak + \Xi\, \calI_2 + \beta\, \calIa & - \oI\bVec^* \\
  - \bVec\, \oI^\ast & \AMat + \beta\, \IMat
  \end{array}\right]. 
\end{equation}
Due to $\calI = \calI_1+\calI_2$ we have 
\begin{align*}
  \Ak + \Xi\, \calI_2 + \beta\, \calIa
  &= \Ak + \Xi\, \calI_2 + {\tbeta}\tfrac{\alpha_1}{\aalpha}\calI_1 + {\tbeta} \tfrac{\alpha_2}{\aalpha}\calI_2 \\
  &= \Aktb + \Xi\, \calI_2 + {\tbeta}\big({\aalpha}^{-1}\calIa - \calI \big)
\end{align*}
with $\big({\aalpha}^{-1}\calIa - \calI \big) \ge 0$. The shifted eigenvalue problem has the form $\bigAMat_\beta \zVec = (\lambda + \beta)\, \bigIMat\, \zVec$. Corresponding to the operator $\bigAMat_\beta$, we define $\bigAForm_\beta \colon \calV\times \calV\to \C$ by 
\begin{align*}
  \bigAForm_\beta(\zVec_1, \zVec_2)
  := \big\langle \bigAMat_\beta \zVec_1, \zVec_2\big\rangle  
  = \bigAForm(\zVec_1, \zVec_2) + \beta\, \bigIForm(\zVec_1, \zVec_2)
\end{align*}
for $\zVec_1 = [u; \xVec_1]$, $\zVec_2 = [v; \xVec_2] \in \calV$. 
\begin{lemma}
\label{lem_pos_bigAForm}
The sesquilinear form $\bigAForm\colon \calV\times \calV\to \C$ defined through $\bigAForm(\cdot\, , \cdot) := \langle \bigAMat\, \cdot\, , \cdot\rangle$ satisfies Assumption~\ref{ass_a}.
\end{lemma}
\begin{proof}
Due to the given structure of $\bigAMat$ and the fact that $\AMat$ is Hermitian, it is easy to see that $\bigAForm$ is continuous and Hermitian. It remains to show that $\bigAForm_\beta$ is positive for some shift~$\beta>0$. For this, we consider $\zVec = [u; \xVec] \in \calV$ and note that 
\[
  \bigAForm_\beta(\zVec, \zVec)	
  \ge \big\langle  \Aktb u + \Xi\, \calI_2 u, u \big\rangle - 2 \real\, \langle \oI\bVec^* \xVec, u \rangle + \big((\AMat + \beta\, \IMat) \xVec, \xVec\big)_{\LOmgII^L}.
\]
The definition of $\bVec=[\xi_1\eta_1, \ldots, \xi_L\eta_L]^T$ yields the estimate 
\begin{align*}
  2\real\, \langle \oI\bVec^* \xVec, u\rangle 
  = 2\sum_{\ell=1}^L\real\, (\xi_\ell\eta_\ell x_\ell,  u)_{\LOmgII}
  &\le 2 \sum_{\ell=1}^L \norm{\eta_\ell x_\ell}_\LOmgII \norm{\xi_\ell u}_\LOmgII \\
  &\le \max \eta^2_\ell\,\Vert\xVec\Vert^2_{\LOmgII^L} + \max \xi^2_\ell\, \Vert u\Vert^2_\LOmgII
\end{align*}
and thus, with $\cell$ denoting the ellipticity constant of $\Aktb$,
\begin{align*}
  \bigAForm_\beta(\zVec, \zVec)	
  \ge \cell\, \Vert u\Vert^2_V + (\Xi-\max \xi^2_\ell)\, \Vert u\Vert^2_\LOmgII + (\min \eta_\ell^2 - \max \eta_\ell^2 + \beta)\, \Vert \xVec \Vert^2_{\LOmgII^L}.
\end{align*}
Assuming $\beta > \max \eta_\ell^2 -\min \eta_\ell^2$, we conclude 
the positivity of $\bigAForm_\beta$.
\end{proof}
%
%
\subsection{Convergence of the inverse power method}
We apply the inverse power method from Section~\ref{sect:convHerm:power} to the shifted eigenvalue problem~$\bigAMat_\beta \zVec = (\lambda + \beta)\, \bigIMat\, \zVec$. For an initial function $\zVec^0 = [u^0; \xVec^0] \in\calV$ we thus consider the iteration 
\begin{align}
\label{eq:linearizedEVP:invIteration}
	\bigAMat_\beta \zVec^j
	= \mu^{j-1}\, \bigIMat\, \tilde \zVec^{j-1}
	\qquad\text{in } \V^*
\end{align}
with the Rayleigh quotient 
\[
	\mu^{j}
	:= \frac{ \bigAForm_\beta(\zVec^{j}, \zVec^{j}) }{ \bigIForm (\zVec^{j}, \zVec^{j})} 
	= \frac{ \Vert \zVec^{j} \Vert_{\beta}^2 }{ \Vert \zVec^{j}\Vert^2} 
	\ge 0.
\]
Note that we use here the norms $\Vert \zVec \Vert_{\beta}^2 = \bigAForm_\beta(\zVec, \zVec)$ and $\Vert \zVec \Vert^2 = \bigIForm(\zVec, \zVec)$ in line with Section~\ref{sect:formulation:setting}. Further, $\tilde \zVec^{j-1}$ denotes the normalization in the $\bigIForm$-norm, i.e., we define $\tilde \zVec^{j-1} := \zVec^{j-1} / \Vert \zVec^{j-1}\Vert$. Note, however, that in the present setting the embedding $\calV \hook \calH$ is not compact and thus, Assumption~\ref{ass_compact} is not satisfied. Hence, Theorem~\ref{thm:power} is not applicable but we are able to show the following result.
\begin{theorem}
Consider the nonlinear eigenvalue problem~\eqref{eq:nonlinearWeak:realCase} for a fixed wave vector~$\k$ and a starting function $u^0 \in V$. Set~$\zVec^0 := [u^0; \xVec^0] \in\calV$ with any~$\xVec^0\in \LOmgII^L$. Then, the power iteration~\eqref{eq:linearizedEVP:invIteration} converges in the sense that there exists a subsequence of $u^j$, which satisfies $u^j \conv \ulim$ in $H$ with $\ulim$ being an eigenfunction of~\eqref{eq:nonlinearWeak:realCase}. 
\end{theorem}
\begin{proof}
We proceed as in the proof of Theorem~\ref{thm:power} and test~\eqref{eq:linearizedEVP:invIteration} with $\zVec^{j}$ and $\tilde \zVec^{j-1}$, respectively. This then yields the estimates
\[
  \Vert \zVec^j\Vert^2_\beta 
  \le \mu^{j-1} \Vert \zVec^j \Vert, \qquad
  \sqrt{\mu^{j}} \le \sqrt{\mu^{j-1}}.
\]
Due to $\mu^j\ge 0$, we conclude the existence of a limit $\mulim := \lim_{j \to \infty} \mu^j \ge 0$. This also implies the uniform bounds 
\[
  \Vert \tilde \zVec^j\Vert_\beta
  = \sqrt{\mu^j}
  \le \sqrt{\mu^0}, \qquad
  \Vert \zVec^j\Vert_\beta 	
  \lesssim C_{\calV \hook \calH}\, \mu^{j-1}
  \le C_{\calV \hook \calH}\, \mu^{0}. 
\]
For the last estimate we have used the continuity of the embedding $\calV \hook \calH$ and the ellipticity of~$\bigAForm_\beta$ shown in Lemma~\ref{lem_pos_bigAForm}. Thus, there exist convergent subsequences and limits $\zlim, \zlimt \in \calV$, which satisfy (without relabeling) $\zVec^j \wconv  \zlim, \tilde\zVec^j \wconv \zlimt$ in $\calV$. We emphasize that the two limits can only differ by a multiplicative constant, i.e., $\zlimt = c\, \zlim$. A componentwise consideration with $\zVec^j = [ u^j; \xVec^j]$, $\tilde \zVec^j = [\tilde{u}^j; \tilde{\xVec}^j]$, and $\zlim = [\ulim; \xlim]$ then yields  
\[
  u^j \wconv \ulim \text{ in }V, \quad
  \tilde u^j \wconv c\, \ulim\text{ in }V, \qquad
  \xVec^j \wconv \xlim \text{ in }\LOmgII^L, \quad
  \tilde \xVec^j \wconv c\, \xlim \text{ in }\LOmgII^L.
\]
Using the compact embedding $V\hook H$, we conclude that the first component converges strongly in $H$, i.e., $u^j \conv \ulim$ and $\tilde u^j \conv c\, \ulim$ in $H$. We now show that the limit pair $(\ulim, \llim)$ with $\llim := c\,\mulim-\beta$ solves the nonlinear eigenvalue problem~\eqref{eq:nonlinearWeak:realCase}. For this, we apply to~\eqref{eq:linearizedEVP:invIteration} a test function $[v;0] \in \calV$ with arbitrary $v\in V$ and consider the limit $j \to \infty$,
\[
  \bigAForm_\beta(\zlim, [v;0]) 
  \leftarrow \bigAForm_\beta(\zVec^j, [v;0]) 
  = \mu^{j-1}\, \bigIForm(\tilde \zVec^{j-1}, [v;0])
  \to \mulim\, \bigIForm(c\, \zlim, [v;0]).
\] 
By the definitions of $\bigAForm_\beta$ and $\bigIForm$ we conclude that
\[
 \Ak \ulim + \Xi\, \calI_2 \ulim 
  = c\, \mulim\, \calIa \ulim - \beta\, \calIa \ulim + \oI \langle \bVec, \xlim\rangle
  \qquad\text{in } V^*.
\]
On the other hand, taking the limit in~\eqref{eq:linearizedEVP:invIteration} with a test function $[0;\yVec] \in \calV$, we obtain 
\[
  \big(\AMat + (\beta - c\,\mulim)\, \IMat\big) \xlim
  = j_\LOmgII^{-1} \bVec\, \oI^\ast \ulim
  \qquad\text{in } \LOmgII^L.	
\] 
These two equations together then give 
\[
 \Ak \ulim + \Xi\, \calI_2 \ulim 
  = \llim\, \calIa \ulim + \bVec^* (\AMat - \llim\, \IMat )^{-1} \bVec\, \calI_2 \ulim  
  \qquad\text{in } V^*.
\]
Thus, the pair $(\ulim, \llim)$ is an eigenpair of the nonlinear eigenvalue problem~\eqref{eq:nonlinearWeak:realCase}. 
\end{proof}
This result shows that the convergence in the eigenfunction is maintained for this particular case coming from the linearization of a nonlinear eigenvalue problem. In contrast to Section~\ref{sect:convHerm}, however, we only showed the convergence in $H$ and not in $V$ due to the lack of compactness.  	   	
\section{A Newton Method for Nonlinear Eigenvalue Problems Arising in Photonic Crystals}\label{sect:nonlinear} 
Similarly as in Section \ref{sect:nonlinearHermit}, we consider an extension of Example~\ref{exp:TM} where the electric permittivity $\eps$ is frequency-dependent and has the form \eqref{eq:distributionPermittivity}. Here, however, we allow for a more general class of models for the electric permittivity $\eps_2(\omega)$. Indeed, the simplified Drude-Lorentz model considered in Section~\ref{sect:nonlinearHermit} has the particular property that it can be linearized in a way that ensures Assumption~\ref{ass_a}.  But other -- more realistic -- models may not have this nice property, if complex poles exist in $\eps_2(\omega)$. 
Indeed, this rules out the possibility to obtain a realization analogous to \eqref{eq:realization:realCase} where $\AMat$ and $\IMat$ are Hermitian matrices and $\IMat$ is positive definite. Because of this, we follow here an alternative strategy and tackle the nonlinear problem directly by a Newton method. 

We are especially interested in the general causality-preserving model described in \cite{GarDZ17} and the full Drude-Lorentz model, see e.g.~\cite{LuoL10} or \cite[Ch.~7.5]{Jac99}, i.e.,
\[
  \eps_\text{CP}(\omega) = 1+\sum\limits_{\ell=1}^L \frac{A_\ell}{\omega-B_\ell}-\frac{\overline{A_\ell}}{\omega + \overline{B_\ell}} \quad \text{ and } \quad \eps_\text{DL}(\omega)=\alpha_\text{DL} + \sum\limits_{\ell=1}^L \frac{\xi_\ell^2}{\eta_\ell^2-\omega^2-i\omega\gamma_\ell}
\]
for some empirically determined parameters $A_\ell, B_\ell$ and, respectively, $\alpha_\text{DL}$, $\xi_\ell$, $\eta_\ell$, $\gamma_\ell$ for $\ell = 1,\dots, L$. 
Here we restrict ourselves to the real part of these models so that the resulting eigenvalue problem \eqref{eq:nonlinearEVP} is self-adjoint, i.e., we consider functions of the form
\[
  \eps_\text{CP}^{\Re}(\omega) = 1+\sum\limits_{\ell=1}^L \frac{2(\omega^2-|B_\ell|^2)\Re(A_\ell\overline{B_\ell})-4\omega^2\Im(A_\ell)\Im(B_\ell)}{(\omega^2-|B_\ell|^2)^2+4\omega^2\Im(B_\ell)^2},
\]
and
\begin{equation}
\label{eq:realDrudeLorentzModel}
  \eps_\text{DL}^{\Re}(\omega)=\alpha_\text{DL}+\sum\limits_{\ell=1}^L \frac{\xi_\ell^2(\eta_\ell^2-\omega^2)}{(\eta_\ell^2-\omega^2)^2+\gamma_\ell^2\omega^2}.
\end{equation}
The information about the imaginary part can be included later through perturbation theory as a post-processing step, see \cite{RamF11}. 
After discussing the properties of the resulting nonlinear eigenvalue problem in details, we show how to solve it with a Newton iteration, provided some a priori knowledge on the eigenpair of interest is given. 
First, however, we discuss the application of Newton's method in the linear case. 
%
%
\subsection{Newton's method for linear eigenvalue problems}
Even for the linear case $\calA u = \lambda\, \calI u$ one may apply Newton's method, cf.~\cite{MehV04, Sch08}. For this, we rewrite the eigenvalue problem as 
\[
\calF_{\bb}(u, \lambda) 
:= \begin{bmatrix}
(\calA - \lambda\, \calI)\, u \\
(\b, u) - 1
\end{bmatrix}	
= 0
\]
with a fixed function $\y\in\V$, $(\b, u)\neq 0$, serving as normalization. Further, the second equation balances the number of equations and variables. The resulting Newton iteration reads 
\begin{align}
\label{eq:Newton:Linear}
\begin{bmatrix}  \calA - \lambda^{j-1}\, \calI  &  - \calI u^{j-1} \\  
(y, \cdot\,)  & 0  \end{bmatrix}
\begin{bmatrix} u^{j}-u^{j-1} \\ \lambda^{j}-\lambda^{j-1} \end{bmatrix}
=   
-\begin{bmatrix} \calA u^{j-1} - \lambda^{j-1}\calI u^{j-1} \\ (\y, u^{j-1}) - 1 \end{bmatrix}.
\end{align}
This means that all iterates are normalized to $(\y, u^j) = 1$ and $u^j$ satisfies
\[
(\calA - \lambda^{j-1}\, \calI)\, u^j
= (\lambda^j-\lambda^{j-1})\, \calI u^ {j-1}.
\]
The similarities to the power method are obvious. More precisely, Newton's method leads to a {\em shifted inverse iteration} with the shift given by the previous eigenvalue approximation. 
Considering the constant shift zero, we obtain -- up to scaling -- the power iteration given in~\eqref{eq:iter:power}. 
The precise algorithm is given in~\cite[Alg.~I]{MehV04}. In the matrix case, this then leads to a local third-order convergence of the eigenvalue, cf.~\cite{Osb64}.  
%
%
\subsection{Definition of the nonlinear eigenvalue problem}
In the remainder of this section, we assume that $\eps_1(\omega) \equiv \alpha_1 \in \R^+$ and that~$\eps_2$ is a real-valued function which is analytic on an open connected set $S\subset\R$ where it satisfies
\[
0<C_0\leq |\eps_2(\omega)|\leq C_1<\infty \qquad \text{and} \qquad \eps_2(-\omega)=\eps_2(\omega)
\]
for all $\omega\in S$. 

Given $\V = \Hper$, $\cH= L^2(\Omega)$ with the standard $H^1$ and $L^2$-norm, respectively, we consider the nonlinear eigenvalue problem
\begin{equation}
\label{eq:nonlinearEVP}
  \calT(\omega)u\coloneqq \calA u - \calB(\omega) u=0  \qquad\text{ in } \V^*
\end{equation}
where $\calA\colon \V\to \V^*$ is defined by the sesquilinear form \eqref{eq:formA_TM} via $\langle \calA u, v \rangle_{\V^*, \V}=a(u,v)$ for all $u,v \in \V$ and $\calB(\omega)\colon S \to \calL(\V, \V^*)$ is defined by
\[
  \langle \calB(\omega) u, v \rangle_{\V^*,\V}
  := \omega^2 \int_{\Omega} \big(\alpha_1 \chi_1(x)+\eps_2(\omega)\chi_2(x) \big)\, u(x) \overline{v(x)}\dx.
\]
Here, $\calL(X, Y)$ denotes the set of bounded linear operators mapping from $X$ to $Y$.

An eigenvalue problem similar to \eqref{eq:nonlinearEVP}, but formulated in terms of operators mapping from $\V$ into itself, was already analyzed in~\cite{Eng10}. Some properties of the operator-valued function $\calT(\omega)$ can be directly derived from there. 
\begin{lemma}
\label{lem:extensionOfTDomain}
The operator-valued function $\calT(\omega)$ can be extended to a self-adjoint holomorphic operator-valued function  in some neighborhood $D{\subset \C} $ of $S$, where $D$ is symmetric with respect to the real axis. Moreover, the spectrum of $\calT(\omega)$  consists of isolated eigenvalues of finite multiplicity.
\end{lemma}
\begin{proof}
See the proof of Lemma 4.4 in \cite{Eng10} and the following remark.
\end{proof}
\begin{lemma}
\label{lem:Fredholm}
For all $\omega \in D$ the operator $\calT(\omega)\colon \V\to\V^*$ is Fredholm with index 0.
\end{lemma}
\begin{proof}
In \cite{Eng10} it is shown that the operator $T(\omega)\colon \V \to \V$ defined by $(T(\omega)u, v)_\V := \langle \calT(\omega) u, v \rangle_{\V^*, \V}$ for all $u, v\in\V$ is Fredholm with index 0. Moreover, the Riesz isomorphism $j_{\V}$ between $\V$ and $\V^*$ is a bounded and bijective linear operator, thus it is Fredholm. Further, its index is 0 since $\dim(\ker(j_\V))=\text{codim}(\text{range}(j_\V))=0$. It follows that $\calT(\omega)=j_\V\, T(\omega)$ is a composition of two Fredholm operators with index $0$ and thus, $\calT(\omega)$ is Fredholm with index $0$ itself. 
\end{proof}
We denote by $(\ulim, \omlim) \in \V\times S$ an eigenpair of~\eqref{eq:nonlinearEVP} and assume the multiplicity of the eigenvalue $\omlim$ to be $1$.
%
%
\subsection{A Newton iteration for the nonlinear eigenvalue problems}
In this subsection, we define a Newton iteration. 
We will prove convergence to the eigenpair $(\ulim, \omlim)$ under certain conditions discussed in the next subsection. 
For this, we extend the results of \cite[Ch.~4]{Sch08} to the infinite-dimensional case and therefore define the iteration following a similar reasoning. We start with rewriting the eigenvalue problem~\eqref{eq:nonlinearEVP} as 
\begin{equation}
\label{eq:defNewtonFunction}
  \calF_{\bb}(u, \omega)
  :=\left[\begin{array}{c} \calT(\omega)u\\ \calP_{\bb}u-1 \end{array} \right]
  = 0,
\end{equation}
where the functional $\calP_{\bb}\colon\V\to\C$ is defined by $\calP_{\bb}u\coloneqq (\b, u)$ for all $u\in\V$ and the normalizing vector $\b \in \V$ has to be chosen such that $\calP_{\bb}\ulim\neq 0$. Note that the self-adjointness of $\calT(\omega)$ for all $\omega\in S$ implies that the eigenvalues of $\calT$ are real. However, since $\V$ is defined over the field $\C$, the well-posedness of the Newton iteration requires an extension of the domain of definition of $\calT$ to an open and convex neighborhood $D\subset\C$ of $S$, cf.~Lemma~\ref{lem:extensionOfTDomain}.

As $\calF_{\bb}$ depends linearly on the first argument $u$ and depends on the second argument $\omega$ only through the holomorphic function $\mathcal{T}$, it follows that $\calF_{\bb}$ is thrice continuously Fr{\'e}chet differentiable on $E\coloneqq \G \times D$ where $\G$ is an open and convex subspace of $\V$ containing $\ulim$.  A Taylor expansion of $\calF_{\bb}$ around~$(u, \omega)\in E$ yields
\begin{align*}
  \calF_{\bb}(\ulim, \omlim) 
  =&\ \calF_{\bb}(u, \omega) + \partial \calF_{\bb}(u, \omega) \left[\begin{array}{@{}c@{}}
\ulim-u\\ \omlim-\omega
\end{array}\right] 
  \\&\quad+ \frac{1}{2}\partial^2 \calF_{\bb}(u, \omega)\left(\left[
\begin{array}{@{}c@{}}
\ulim-u\\ \omlim-\omega 
\end{array}
\right], \left[\begin{array}{@{}c@{}} \ulim-u\\ \omlim-\omega \end{array}\right]\right) +\calO\left(\left\Vert \left(\begin{array}{@{}c@{}} \ulim-u\\ \omlim-\omega \end{array}\right)\right\Vert^3_\VxC\right)
\end{align*}
where $\Vert \left( u, \ \omega\right) \Vert_\VxC^2 \coloneqq \Vert u\Vert_{\V}^{2}+|\omega|^2$ for all $(u, \omega) \in \VxC\coloneqq \V \times \C$. The corresponding Jacobian is given by
\[
  \partial \calF_{\bb}(u, \omega) 
  = \left[\begin{array}{cc}
  \calT(\omega)&{\calT'}(\omega)u\\
  \calP_{\bb}& 0
  \end{array}\right],
\]
where $\calT'$ denotes the derivative of $\calT$ with respect to $\omega$. In order to shorten the notation, we introduce the variables $\Delta u\coloneqq \ulim-u$ and $\Delta \omega\coloneqq \omlim-\omega$. The second Fr{\'e}chet derivative of $\calF_{\bb}$ along the direction $(\Delta u,\, \Delta \omega)$ reads
\[
 \partial^2\calF_{\bb}(u, \omega) \Big(\left[\begin{array}{@{}c@{}} \Delta u\\ \ \Delta \omega\end{array}\right], \left[\begin{array}{@{}c@{}} \Delta u\\ \ \Delta \omega\end{array}\right] \Big) 
 = \Delta\omega\left[\begin{array}{c} 2{\calT'}(\omega)\Delta u  + {\calT''}(\omega)u \Delta\omega \\0
\end{array}\right].
\]
Following the strategy of the standard formulation of Newton's method, we define the following iterative process: Given an initial vector $(u^0, \, \omega^0) \in E$ satisfying $\calP_{\bb}u^0= 1$, the iterates are defined by $u^{j+1}:=u^j+s^j$, $\omega^{j+1}:=\omega^j+\nu^j$ with
\begin{equation}
\label{eq:1stNewtonIt}
 \left[\begin{array}{cc}
\calT(\omega^j)&{\calT'}(\omega^j)u^j\\
\calP_{\bb}& 0
\end{array}\right]   \left[\begin{array}{c} s^{j} \\ \nu^j \end{array}\right]  = - \left[\begin{array}{c} \calT(\omega^j)u^j \\ 0 \end{array}\right]
\end{equation}
for $j=0,1,2,\ldots$. 

We show in the following subsection that \eqref{eq:1stNewtonIt} is uniquely solvable for $(s^j, \, \nu^j)$ when $(u^j,\, \omega^j)$ lies in a neighbourhood of the eigenpair~$(\ulim,\, \omlim)$.
%
%
\subsection{Convergence of the Newton iteration}
This subsection is devoted to the proof of local convergence of the iteration introduced in~\eqref{eq:1stNewtonIt}. We adopt the ideas from \cite{AnsR68}, where operators mapping from $\V$ to $\V$ are considered. For $v\in \V$ and $\omega \in \C$ we denote by $B_\tau(v, \omega)\subset \VxC$ the open ball of radius $\tau$ in the $\VxC$-norm centred at $(v, \omega)$ and by $B_\tau(v)\subset \V$  the open ball of radius $\tau$ in the $\cH$-norm centred at $v$. We base the proof on the Kantorovich theorem as stated and proved in \cite[Th.~2.1]{Deu04}.
\begin{theorem}
\label{thm:Kantorovitch}
Let $\calF_{\bb}\colon E\to \VxC^*$ be the continuously Fr{\'e}chet differentiable mapping defined in \eqref{eq:defNewtonFunction}. For a starting point $z^0=(u^0, \omega^0)\in E$ let $\partial\calF_{\bb}(z^0)$ be invertible. Further, we assume that~$\overline{B_{\rho}(z^0)}\subset E$ with~$\rho\coloneqq (1-\sqrt{1-2h_0})/\kappa_0$ and~$h_0\coloneqq \alpha \kappa_0\leq0.5$, where~$\alpha, \kappa_0 < \infty$ are defined through
\begin{align*}
\Vert\partial\calF_{\bb}(z^0)^{-1}\calF_{\bb}(z^0)\Vert_{\VxC}&\leq \alpha,\\
\Vert\partial\calF_{\bb}(z^0)^{-1}(\partial\calF_{\bb}(z)-\partial\calF_{\bb}(\tilde{z})) \Vert_{\VxC} &\leq \kappa_0\, \Vert z-\tilde{z}\Vert_{\VxC}
\end{align*}
for all $z, \tilde{z} \in E$. Then, the sequence $(u^j, \omega^j)$ obtained from iteration \eqref{eq:1stNewtonIt} is well-defined, remains in $\overline{B_{\rho}(z^0)}$, and converges to some $(\ulim, \omlim)$ with $\calF_{\bb}(\ulim, \omlim)=0$. For $h_0<0.5$, the convergence is quadratic.
\end{theorem}
In order to use Theorem~\ref{thm:Kantorovitch}, we need to determine under which conditions the Jacobian $\partial\calF_y(z^0)$ is invertible.

\begin{lemma}
\label{lem:decompositionOfVStar}
There exists $\tau>0$ such that 
\[
\langle \calT'(\omlim) u, \ulim \rangle_{\V^*,\V} \neq 0
\]
and
\[\V^*=\ker(\calQ_{[\ulim]}) \oplus \sspan\{\calT'(\omlim)u\}\]
hold for any $u\in B_{\tau}(\ulim)$, where $\calQ_{[\ulim]}\colon \V^*\to \C$ is defined by
$\calQ_{[\ulim]} := \langle\, \cdot\, , \ulim \rangle_{\V^*, \V}$. 
\end{lemma}
\begin{proof}
By assumption the (geometric) multiplicity of the eigenvalue $\omlim$ is 1, which means that $\mbox{dim}\big(\ker(\calT(\omlim))\big)=1$.  
From Lemma~\ref{lem:Fredholm} we know that $\calT(\omlim)$ is Fredholm with index 0, which implies \[
\mbox{codim}\big(\range(\calT(\omlim))\big)=\mbox{dim}\big(\ker(\calT(\omlim))\big)=1.
\] 
Further, the algebraic multiplicity of the eigenvalue $\omlim$ is also 1 by assumption, thus the Jordan chains of $\calT$ at $\omlim$ are all of length 1 and in particular $\calT'(\omlim) \ulim \not\in \range(\calT(\omlim))$, see e.g. \cite[Chap. 7]{LopM07}. 
The last statement stays true in a neighbourhood of $\ulim$ since $\calT'(\omlim)$ is a linear operator, i.e.,~there exists $\tau>0$ such that $\calT'(\omlim) u \not\in \range(\calT(\omlim))$ for all $u \in B_{\tau}(\ulim)$. It follows that we can decompose $\V^*=\range(\calT(\omlim)) \oplus \sspan\{\calT'(\omlim)u\}$. 

In the remainder of this proof, let $u$ be any vector in $ B_{\tau}(\ulim)$. 
We note that the hermiticity of $\calT(\omlim)$ implies that $\langle \calT(\omlim)v, \ulim \rangle_{\V^*, \V}=\langle \calT(\omlim)\ulim, v \rangle_{\V^*, \V}=0$ for all $v\in\V$.
Now, the eigenfunction $\ulim\not\equiv 0$ per definition, therefore there exist some 
$h\in\V^*$, $\alpha\in\C$, and $v\in\V$ with $h = \alpha\calT'(\omlim)u + \calT(\omlim) v$ such that \[0\neq \langle h , \ulim\rangle_{\V^*, \V} = \langle \alpha\calT'(\omlim)u , \ulim\rangle_{\V^*, \V} + \langle \calT(\omlim) v , \ulim\rangle_{\V^*, \V} =\alpha \langle \calT'(\omlim)u , \ulim\rangle_{\V^*, \V}.\] Therefore, $\langle \calT'(\omlim)u , \ulim\rangle_{\V^*, \V}\neq 0$.

It remains to show that $\ker(\calQ_{[\ulim]}) = \range(\calT(\omlim))$. 
We start by showing the inclusion "$\subseteq$".  Given any $h\in\ker(\calQ_{[\ulim]})\subset\V^*$, we can write $h = \alpha \calT'(\omlim)u + \calT(\omlim)v$ for some $\alpha \in \C$ and $v\in\V$. 
Thus, $0=\calQ_{[\ulim]}(h)=\langle \alpha \calT'(\omlim)u, \ulim \rangle_{\V^*, \V}+\langle \calT(\omlim)v, \ulim \rangle_{\V^*, \V}=\alpha \langle \calT'(\omlim)u, \ulim \rangle_{\V^*, \V}$ implies $\alpha=0$ and therefore $h\in\range(\calT(\omlim))$.
The reverse inclusion follows from the fact that for any $h\in\range(\calT(\omlim))$ there exists some $v\in \V$ such that $h=\calT(\omlim)v$ and thus, 
$\calQ_{[\ulim]}(h)=\langle h, \ulim \rangle_{\V^*, \V}=\langle \calT(\omlim)v, \ulim \rangle_{\V^*, \V} = 0$.
\end{proof}

\begin{lemma}
\label{lem:invertibilityOfJacobian2}
Consider~$u\in B_{\tau}(\ulim)$ with $\tau$ as in Lemma \ref{lem:decompositionOfVStar}. Then $\partial\calF_{\bb}(u, \omlim)$  has a bounded inverse.
\end{lemma}

\begin{proof}
We prove the bijectivity of $\partial\calF_{\bb}(u, \omlim)$ by showing that the equation \begin{equation}
\label{eq:showBijectivity}
 \partial \calF_{\bb}(u, \omega) \left[ \begin{array}{c} s \\ \nu \end{array}\right]= \left[ \begin{array}{c} f \\ g \end{array}\right]
\end{equation}
with $\omega=\omlim$ and $u \in B_{\tau}(\ulim)$  has a unique solution for all $f\in\V^*$, $g\in\C$. 
Since $\calP_{\bb} \ulim\neq 0$ by assumption, we can decompose 
\begin{equation}
\label{eq:decompositionOfV}
\V = \sspan\{\ulim\} \oplus \ker(\calP_{\bb}).
\end{equation}
Indeed, any $v\in\V$ can be written as the sum $v= \alpha_v \ulim + v_{\ker}$ with $\alpha_v\coloneqq \left(\calP_{\bb}\ulim\right)^{-1}\calP_{\bb}v \in \C$ and $v_{\ker}=v -\alpha_v \ulim \in \ker(\calP_{\bb})$. Therefore, the lower row of \eqref{eq:showBijectivity} implies $g=\calP_{\bb} s = \alpha_s \calP_{\bb}\ulim$ and thus $\alpha_s=(\calP_{\bb}\ulim)^{-1}g$. 
Using the self-adjointness of $\calT(\omlim)$, we find that the upper row of \eqref{eq:showBijectivity} tested with $\ulim$ yields
\[
  \langle f, \ulim\rangle_{\V^*, \V}= \langle \calT(\omlim) s, \ulim\rangle_{\V^*, \V} + \langle \nu \calT'(\omlim) u, \ulim\rangle_{\V^*, \V}= \nu \langle  \calT'(\omlim) u, \ulim\rangle_{\V^*, \V},
\]
from which it follows that 
\[
  \nu=\frac{\langle f, \ulim\rangle_{\V^*, \V}}{\langle  \calT'(\omlim) u, \ulim\rangle_{\V^*, \V}},
\]
since the condition $\langle  \calT'(\omlim) u, \ulim\rangle_{\V^*, \V}\neq 0$ holds for all $u\in B_{\tau}(\ulim)$.
Moreover, we note that the restriction of $\calT(\omlim)$ to $\ker(\calP_{\bb})$ is bijective as mapping from $\ker(\calP_{\bb})$ to its range. Indeed, 
\[
\ker(\calT(\omlim)|_{\ker(\calP_{\bb})})= \ker(\calT(\omlim)) \cap \ker(\calP_{\bb}) = \{0\}
\] 
implies the injectivity and the surjectivity is obvious from the definition. Therefore we obtain from the upper row of \eqref{eq:showBijectivity} that
\[
f=\calT(\omlim) s + \nu \calT'(\omlim)u= \calT(\omlim)|_{\ker(\calP_{\bb})} s_{\ker} +\nu \calT'(\omlim)u,
\]
which implies
\[
  s_{\ker}=\calT(\omlim)|_{\ker(\calP_{\bb})}^{-1}\big( f-\nu \calT'(\omlim)u \big)
\]
if $f-\nu \calT'(\omlim)u \in \range(\calT(\omlim)|_{\ker(\calP_{\bb})})$. In order to show that this last condition is satisfied, we use the fact that $\ker(\calQ_{[\ulim]})=\range(\calT(\omlim))$ and the decomposition \eqref{eq:decompositionOfV}. This finally implies $f-\left(\langle  \calT'(\omlim) u, \ulim\rangle_{\V^*, \V}\right)^{-1}\langle f, \ulim\rangle_{\V^*, \V} \calT'(\omlim)u \in \ker(\calQ_{[\ulim]})=\range(\calT(\omlim)|_{\ker(\calP_{\bb})})$.
\end{proof}
\begin{lemma}
\label{lem:invertibilityOfJacobian}
Let $\omega\neq \omlim \in D$ be close enough to $\omlim$ so that $\calT(\omega)$ is invertible,  and let $u\in\V$ satisfy $\calP_{\bb} \calT(\omega)^{-1}\calT'(\omega)u\neq 0$. Then $\partial \calF_{\bb}(u, \omega)$ has a bounded inverse.
\end{lemma}
\begin{proof}
Note that there exists a neighborhood around $\omlim$ in which $\calT(\omega)$ is invertible, since the spectrum of $\calT$ consists of isolated eigenvalues. We prove the bijectivity of $\partial\calF_{\bb}(u, \omega)$ by showing that \eqref{eq:showBijectivity} has a unique solution for all $f\in\V^*, g\in\C$.

The upper row of \eqref{eq:showBijectivity} implies
\[
  s= \calT^{-1}(\omega)(f- \nu\, \calT'(\omega)u )
\]
whereas the lower row implies
\[
  g=\calP_{\bb} s= \calP_{\bb}\calT^{-1}(\omega)(f- \nu \calT'(\omega)u ).
\]
Thus, $\nu\in\C$ is uniquely defined by 
\[
  \nu = \frac{\calP_{\bb}\calT^{-1}(\omega)f-g}{\calP_{\bb}\calT^{-1}(\omega)\calT'(\omega)u}
\] 
as long as the condition $\calP_{\bb}\calT^{-1}(\omega)\calT'(\omega)u\neq 0$ is satisfied. As a result, $\partial \calF_{\bb}(u, \omega)$ is linear, bijective, and bounded, which implies the existence of a bounded inverse. 
\end{proof}
The previous lemmata show that choosing $(u^0,\, \omega^0)$ close enough to $(\ulim, \omlim)$ in $E$ with
\begin{equation}
\label{eq:conditionStartingPoint}
\calP_{\bb}\calT^{-1}(\omega^0)\calT'(\omega^0)u^0\neq 0 \quad \text{ if } \quad \omega^0\neq \omlim
\end{equation} 
insures that $\partial \calF(u^0, \omega^0)$ has a bounded inverse. In this case, the first iteration step of \eqref{eq:1stNewtonIt} is well-defined. An estimate for the bound $\kappa_0$ appearing in Theorem~\ref{thm:Kantorovitch} can be derived from the fact that $\calF_{\bb}$ is twice continuously Fr{\'e}chet differentiable  and that $\partial^2\calF_{\bb}(u, \omega)$ is bounded for all $(u, \omega)\in E$, since
\[
\begin{aligned}
\Vert\partial^2\calF_{\bb}(u, \omega)\Vert_{\calL(E\times E,\VxC^*)}
&\coloneqq \sup\limits_{z_1, z_2 \in E} \frac{\Vert\partial^2\calF_{\bb}(u, \omega) (z_1, z_2)\Vert_{\VxC^*}}{\Vert z_1\Vert_{\VxC} \Vert z_2\Vert_\VxC}\\
& \leq 2\Vert\calT'(\omega)\Vert_{\calL(\V, \V^*)} + \Vert\calT''(\omega)u\Vert_{\V^*}.
\end{aligned}
\]
To sum up, the results of this section show the local convergence of the Newton iteration \eqref{eq:1stNewtonIt} under the assumption that the starting point~$(u^0,\, \omega^0)\in E$ is close enough to the eigenpair and satisfies the condition \eqref{eq:conditionStartingPoint}. 	
\section{Numerical Experiments}\label{sect:numerics}
In this final section we consider the proposed PDE eigenvalue iterations numerically.  
The first two examples in Sections~\ref{sect:numerics:lin} and~\ref{sect:numerics:nonlin}  consider the application of the inverse power method. 
Here, the essential idea is to combine the iterative solver and mesh refinement. 
Since we have proven the convergence on operator level, we can choose the spatial discretization in each step independently. 
Thus, we can start the computations on coarse grids and consider refinements as the eigenvalue iteration goes on, leading to substantial runtime improvements. 
Finally, Section~\ref{sect:numerics:Newton} shows results for the application of the Newton-type iteration, which was analyzed in Section~\ref{sect:nonlinear}.

In all three experiments the computational domain equals the unit square~$\Omega = (0,1)^2$ with a disk of radius $0.3$ in the middle defining~$\Omega_2$. The outer material is air, i.e., $\eps_1(\omega)\equiv \alpha_1 = 1$, whereas the relative electrical permittivity of the inner material varies. The spatial discretization has been performed with the software library {\em Concepts}~\cite{FraL02,ConceptsWeb} using piecewise polynomial finite element basis functions of degree~$2$ on a quadrilateral mesh with curved cells. We refer to~\cite{SchK09, SchK10} for more details on the application of finite elements to eigenvalue problems arising in photonic crystals simulations using {\em Concepts}. 
%
%
\subsection{Inverse power method for linear model}\label{sect:numerics:lin}
We consider the problem described in Example~\ref{exp:TM} for a fixed wave vector~$\k=[\pi/2, \pi]^T$ and a constant model $\eps_2(\omega)\equiv 8$ for the relative permittivity in $\Omega_2$. In order to find the smallest eigenvalue, we take as shift~$\beta=1$ and apply the Rayleigh quotient iteration~\eqref{eq:iter:power}. 

The obtained convergence history for the eigenvalue is shown in the top plot of Figure~\ref{fig:linear:conv}. More precisely, we show the relative error~$|\mu^j-{\protect \muref}|/|{\protect \muref}|$, where~$\muref$ denotes the reference solution (as approximation of~$\mulim$) obtained on a finer spatial mesh with a large number of inverse iteration steps. Here, the reference mesh corresponds to the sixth level of uniform refinement. Note that~$\muref > \mulim$ due to the spatial discretization error. 
We compare the results for~$3$, $5$, and $7$ iteration steps per mesh. This means that we perform a fixed number of iteration steps before we consider a uniform refinement of the mesh. When the fine mesh (corresponding here to the fifth level of refinement) is reached, the inverse power iterations continue without further mesh refinement. The corresponding convergence on this fine mesh is depicted by the dashed line in Figure~\ref{fig:linear:conv}. 
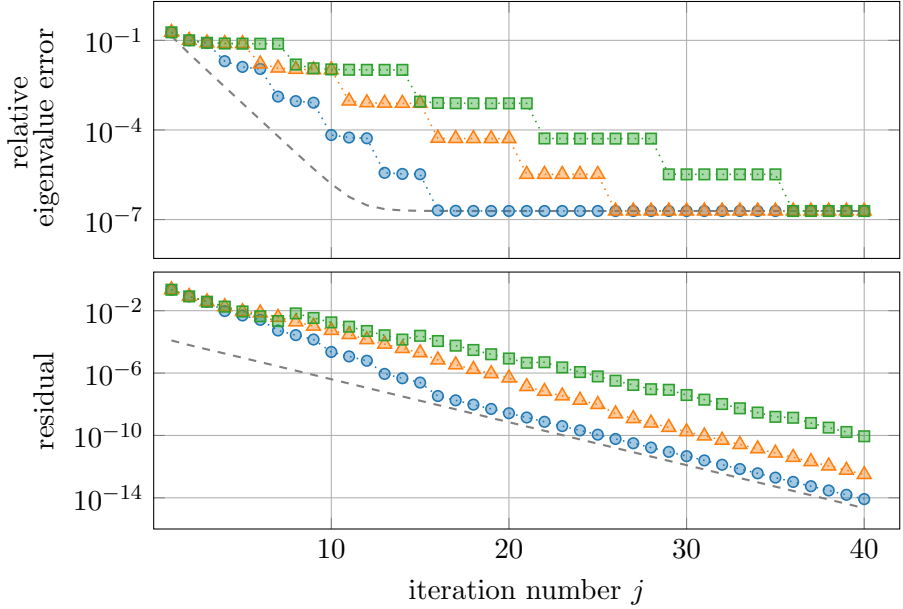
\begin{figure}
\begin{tikzpicture}
\begin{axis}[
width=11.5cm,
height=5cm,
ylabel={$\begin{aligned}&\qquad\text{relative}\\[-0.6em] &\text{eigenvalue error}\end{aligned}$},
xmin=0, xmax=42,
ymin=5.e-9, ymax=2,
ymode=log,
ylabel near ticks,
xtick = {10, 20, 30, 40},
xticklabels = {,,,,},
xmajorgrids,
x grid style={lightgray!92!black},
ytick = {1e-1, 1e-4, 1e-7},
ymajorgrids,
y grid style={lightgray!92!black}]
%
\addplot [semithick, color0, dotted, mark=*, mark size=2, mark options={solid,fill opacity=0.4}]
table {%
1 0.185826449978765
2 0.100372876106732
3 0.082652812028926
4 0.0200609648744385
5 0.0129159752888577
6 0.0110011615043542
7 0.00130959281714915
8 0.000922552089949104
9 0.000817665606795839
10 6.76454425724959e-05
11 5.58197490490863e-05
12 5.25417875934345e-05
13 3.66945012938705e-06
14 3.37259514366301e-06
15 3.28945005386341e-06
16 2.03137511939131e-07
17 1.96188967179214e-07
18 1.94236502460002e-07
19 1.93687542821225e-07
20 1.93533153913339e-07
21 1.9348972143378e-07
22 1.93477505860473e-07
23 1.93474075922652e-07
24 1.93473102993103e-07
25 1.93472838687903e-07
26 1.93472748994979e-07
27 1.93472739070972e-07
28 1.93472732287473e-07
29 1.93472728770104e-07
30 1.93472726634558e-07
31 1.93472729398205e-07
32 1.9347273831725e-07
33 1.93472730277548e-07
34 1.93472729775066e-07
35 1.93472729021344e-07
36 1.93472727011419e-07
37 1.93472732161853e-07
38 1.93472726885798e-07
39 1.93472729649446e-07
40 1.93472722740327e-07
};
\addplot [semithick, color1, dotted, mark=triangle*, mark size=3, mark options={solid, fill opacity=0.4}]
table {%
1 0.185826449978765
2 0.100372876106732
3 0.082652812028926
4 0.0784166323857641
5 0.0773917028112994
6 0.0162369352024553
7 0.011810429631089
8 0.0106899808636894
9 0.0103850192458063
10 0.0103010551727703
11 0.000931217673541885
12 0.000816208967085901
13 0.000787787822289832
14 0.000780102318968278
15 0.000777980408266764
16 5.27161094706512e-05
17 5.16422994934839e-05
18 5.13697934525666e-05
19 5.12954004832481e-05
20 5.12747598558394e-05
21 3.26592197405154e-06
22 3.25934190026089e-06
23 3.25762829621555e-06
24 3.25715688264205e-06
25 3.25702563036316e-06
26 1.93518505074006e-07
27 1.93484992581621e-07
28 1.93476125041673e-07
29 1.93473672555734e-07
30 1.93472997346395e-07
31 1.93472794972024e-07
32 1.93472746482572e-07
33 1.93472736307324e-07
34 1.93472726885798e-07
35 1.93472721609744e-07
36 1.93472727011419e-07
37 1.93472727262659e-07
38 1.93472730403168e-07
39 1.93472728770104e-07
40 1.93472727890761e-07
};
\addplot [semithick, color2, dotted, mark=square*, mark size=2, mark options={solid, fill opacity=0.4}]
table {%
1 0.185826449978765
2 0.100372876106732
3 0.082652812028926
4 0.0784166323857641
5 0.0773917028112994
6 0.0771433390339239
7 0.0770831049468819
8 0.0155157377354382
9 0.0115990234633677
10 0.0106301399470583
11 0.0103682524403855
12 0.0102963786830835
13 0.0102765423337779
14 0.0102710526108279
15 0.000885473902434554
16 0.000803373568043738
17 0.000784187205105084
18 0.00077909145128097
19 0.000777696466188945
20 0.000777309682890559
21 0.000777201750461937
22 5.20022489602464e-05
23 5.14437576982494e-05
24 5.13142781358071e-05
25 5.12798335411125e-05
26 5.12703882346449e-05
27 5.12677657008456e-05
28 5.12670330244714e-05
29 3.26037676465897e-06
30 3.25779608579905e-06
31 3.25719549569719e-06
32 3.2570354371664e-06
33 3.25699150270796e-06
34 3.25697929969668e-06
35 3.25697588948126e-06
36 1.93486834050216e-07
37 1.93476141623558e-07
38 1.9347364931597e-07
39 1.93472980638889e-07
40 1.93472803388587e-07
};
\addplot [thick, gray, dashed]
table {
1 0.140113394342501
2 0.0358007253251508
3 0.00995102455956826
4 0.00277888835768941
5 0.000777759905538272
6 0.000218216946908095
7 6.14112095763769e-05
8 1.73970442083047e-05
9 5.03037466166991e-06
10 1.55375540177945e-06
11 5.76078803064203e-07
12 3.01096179588431e-07
13 2.23747377564114e-07
14 2.01989214758922e-07
15 1.95868504536081e-07
16 1.94146690698899e-07
17 1.93662317377189e-07
18 1.93526065534295e-07
19 1.93487726331506e-07
20 1.93476949236737e-07
21 1.93473915002994e-07
22 1.93473063799556e-07
23 1.93472819719231e-07
24 1.93472757914023e-07
25 1.93472735302361e-07
26 1.93472728518863e-07
27 1.93472726885798e-07
28 1.93472729649446e-07
29 1.93472729272585e-07
30 1.93472725378354e-07
31 1.93472724750253e-07
32 1.93472722237846e-07
33 1.93472719851059e-07
34 1.93472728142002e-07
35 1.9347273103127e-07
36 1.93472722112225e-07
37 1.93472725880836e-07
38 1.93472726508937e-07
39 1.9347272361967e-07
40 1.93472725378354e-07
};
\end{axis}
\begin{axis}[
yshift=-3.6cm,
width=11.5cm,
height=5cm,
xlabel={iteration number $j$},,
ylabel={residual},
xmin=0, xmax=42,
ymin=1.e-16, ymax=3,
ymode=log,
ylabel near ticks,
xtick = {10, 20, 30, 40},
xmajorgrids,
x grid style={lightgray!92!black},
ytick = {1e-2, 1e-6, 1e-10, 1e-14},
ymajorgrids,
y grid style={lightgray!92!black}]
%
\addplot [semithick, color0, dotted, mark=*, mark size=2, mark options={solid, fill opacity=0.4}]
table {%
1 0.221484960861903
2 0.0827753165289429
3 0.0375935852361939
4 0.00946506394141359
5 0.00493935723147978
6 0.00259185687695923
7 0.000538085203566346
8 0.000270888767274281
9 0.00014161196649794
10 2.25800421305903e-05
11 1.1709579863994e-05
12 6.17305741186483e-06
13 8.84360429340744e-07
14 4.64677191347415e-07
15 2.46006531979232e-07
16 3.3730457706028e-08
17 1.78274447722144e-08
18 9.45236919986151e-09
19 5.01693595445998e-09
20 2.66319176110113e-09
21 1.41357422532451e-09
22 7.50181379065015e-10
23 3.98063217640799e-10
24 2.11197357872334e-10
25 1.12043667063581e-10
26 5.94372132311524e-11
27 3.15289056507995e-11
28 1.67241641176126e-11
29 8.8709271372815e-12
30 4.70528063668324e-12
31 2.49572111942084e-12
32 1.32373393647983e-12
33 7.0210176669177e-13
34 3.72391768173602e-13
35 1.97509398258372e-13
36 1.04752991082738e-13
37 5.55566814440155e-14
38 2.94616712802113e-14
39 1.56218717484272e-14
40 8.28263385136748e-15
};
\addplot [semithick, color1, dotted, mark=triangle*, mark size=3, mark options={solid, fill opacity=0.4}]
table {%
1 0.221484960861903
2 0.0827753165289429
3 0.0375935852361939
4 0.0181573009618359
5 0.00890597557606109
6 0.00720549357132703
7 0.00370709690233899
8 0.00193121399837388
9 0.00102185681689574
10 0.000542296975246995
11 0.000294500638484862
12 0.000139401280299951
13 7.12651569310237e-05
14 3.76063571052986e-05
15 2.00161025428083e-05
16 6.93972620149236e-06
17 3.36723112425815e-06
18 1.73115966331441e-06
19 9.13008722500774e-07
20 4.85333015308804e-07
21 1.36030415929596e-07
22 6.66444897419712e-08
23 3.44264711834694e-08
24 1.8172919939263e-08
25 9.65728220763336e-09
26 2.43274228472124e-09
27 1.19953488696445e-09
28 6.21501923061147e-10
29 3.28269107220309e-10
30 1.74413234724912e-10
31 9.2781118416056e-11
32 4.93462197087998e-11
33 2.62312030909206e-11
34 1.39365158870151e-11
35 7.40119375435919e-12
36 3.92919636540962e-12
37 2.0854316542793e-12
38 1.10664057656891e-12
39 5.87160068968701e-13
40 3.11498588746912e-13
};
\addplot [semithick, color2, dotted, mark=square*, mark size=2, mark options={solid, fill opacity=0.4}]
table {%
1 0.221484960861903
2 0.0827753165289429
3 0.0375935852361939
4 0.0181573009618359
5 0.00890597557606109
6 0.00438317201379794
7 0.00215835842493368
8 0.00668554421023508
9 0.00342086227469078
10 0.00177740818061363
11 0.000940312982605482
12 0.00049930752687016
13 0.000264911509127821
14 0.000140321727050614
15 0.00024637839859436
16 0.000113375685169391
17 5.73462715664257e-05
18 3.01927952719651e-05
19 1.60754777021312e-05
20 8.57559326552729e-06
21 4.5717658525111e-06
22 4.95175815228783e-06
23 2.29882158867032e-06
24 1.16148445826052e-06
25 6.09822701704899e-07
26 3.24142066287734e-07
27 1.7277878869536e-07
28 9.20841852703284e-08
29 8.45058978901141e-08
30 3.90062420420405e-08
31 1.97075669346226e-08
32 1.03449993932649e-08
33 5.49736186671598e-09
34 2.92966920141641e-09
35 1.56115312690422e-09
36 1.36680536702649e-09
37 6.28332477508002e-10
38 3.17489776632386e-10
39 1.66655604278281e-10
40 8.85548667616669e-11
};
\addplot [thick, gray, dashed]
table {%
1 0.000123998951860924
2 6.23503408001514e-05
3 3.32444869744427e-05
4 1.78688379095011e-05
5 9.55644616777062e-06
6 5.08759362542443e-06
7 2.70127255870723e-06
8 1.43237676249321e-06
9 7.59134757845057e-07
10 4.02278736715814e-07
11 2.13187134185021e-07
12 1.12992574689523e-07
13 5.98961461632084e-08
14 3.17542959696498e-08
15 1.68364976006586e-08
16 8.92765499265106e-09
17 4.73425010694947e-09
18 2.5106511318791e-09
19 1.33148892681186e-09
20 7.0615597738347e-10
21 3.74517830624973e-10
22 1.98632743952279e-10
23 1.05349816465022e-10
24 5.58753596336542e-11
25 2.9635297082631e-11
26 1.5718096155955e-11
27 8.33665910556533e-12
28 4.42165707711796e-12
29 2.34518679750731e-12
30 1.24385861128495e-12
31 6.59728361498239e-13
32 3.49910727668503e-13
33 1.85587440375055e-13
34 9.84276708125223e-14
35 5.22017340437774e-14
36 2.76837560658183e-14
37 1.46815240608305e-14
38 7.78360211942783e-15
39 4.12318147809444e-15
40 2.18217375776213e-15
};
\end{axis}
\end{tikzpicture}
\caption{The two plots show the relative error in the eigenvalue, i.e., $|\mu^j-{\protect \muref}|/|{\protect \muref}|$ (top) and the residual estimated in the dual norm, i.e., $\Vert\Res(u^j,\lambda^j)\Vert_{\V^*}$ (bottom). In both plots we compare the results for~$3$ ({\protect \blueCirLight{}}), $5$ ({\protect \orangeTriLight{}}), and $7$ ({\protect \greenSquLight{}}) iteration steps per mesh. The dashed line corresponds to computations without refinements on the fine mesh.}
\label{fig:linear:conv}
\end{figure}

One observation is that the results for~$3$ iteration steps ({\protect \blueCirLight{}}) per mesh are comparable to the results obtained on the fine mesh. 
We emphasize that the first $15$ iteration steps are roughly as expensive as a single step on the fine mesh but leading to a much better accuracy. 
Further note that the saturation of the convergence of the eigenvalue is due to the discretization error, since $\muref$ is computed on a finer spatial mesh than the iterates~$\mu^j$. The computational gain is even more distinct for the convergence of the residuum (and thus the eigenfunction), which is shown in the bottom plot of Figure~\ref{fig:linear:conv}. Here, the residuum~$\Res\colon \V\times \C \to \V^*$ is defined by 
\[
  \Res(u, \lambda) 
  \coloneqq \big(\Ak - \lambda \calI\big)u \in \V^*
\]
and measured in the dual norm, i.e., 
\[
  \norm{\Res(u, \lambda)}_{\V^*}
  \coloneqq \sup_{\substack{v\, \in\, \V,\\ \norm{v}_{\V}=1}} \big|\langle \Res(u, \lambda), v \rangle_{\calV^*, \calV} \big|.
\]
In the discrete setting, the dual norm is given by $\Vert \cdot \Vert_{M(K+M)^{-1}M}$. 
The results for $5$ ({\protect \orangeTriLight{}}) and $7$ ({\protect \greenSquLight{}}) iteration steps per mesh exhibit a similar behavior, i.e., a given tolerance is reached faster, since the first iteration steps are performed on coarse meshes, cf.~Table~\ref{tab:runtimes:lin}. 

\begin{table}	
\caption{Runtime comparison (in seconds) until the residual reaches the tolerance tol=$10^{-10}$ in the first numerical experiment.}
\begin{tabular}{r|r|r|r}
	 on fine mesh & $3$ steps per mesh & $5$ steps per mesh & $7$ steps per mesh \\[0.1em] \hline 
	 129.136 & 72.396 & 46.585 & 42.308		\\
\end{tabular}
\label{tab:runtimes:lin}
\end{table}
%
%
\subsection{Inverse power method for nonlinear model}\label{sect:numerics:nonlin}
Motivated by the numerical example in~\cite{EffKE12} we consider a material within the disk $\Omega_2$ having a frequency-dependent permittivity described by a two-term Lorentz model, cf.~\eqref{eq:DrLoModel}, with positive and real parameters 
\[
  \alpha_2 = 2, \qquad
  \xi^2_1 = 98.6960, \qquad
  \xi^2_2 = 197.3921, \qquad 
  \eta^2_1 = 55.2698, \qquad
  \eta^2_2 = 63.1655. 
\]
Recall that we consider here relative permittivities as discussed in Section~\ref{sect:formulation:exp}. As in the previous example, we are interested in the lower-most eigenvalue of the corresponding nonlinear eigenvalue problem~\eqref{eq:nonlinEVP:realCase} for the fixed wave vector~$\k=[\pi/2, \pi]^T$. 

For the numerical solution of the eigenvalue problem, we apply the inverse power method~\eqref{eq:linearizedEVP:invIteration} to the linearized system with shift~$\beta=\eta_2^2-\eta_1^2+1$. This shift guarantees the positivity of the involved operator, cf.~Lemma~\ref{lem_pos_bigAForm}. In Figure~\ref{fig:nonlinear:conv} we compare again the results for a fixed number of iteration steps per mesh. Here, the computational results are even more convincing as the plots for~$3$ and~$7$ iteration steps per mesh are very similar to the results obtained on the fine mesh, although the computational costs are significantly smaller. 
Note that the iterations are based on the linearization~\eqref{eq:linearizedEVP:realCase} but the residuals are computed in terms of the nonlinear eigenvalue problem~\eqref{eq:nonlinEVP:realCase}. 
\begin{figure}
\begin{tikzpicture}
\begin{axis}[
width=11.5cm,
height=5cm,
ylabel={$\begin{aligned}&\qquad\text{relative}\\[-0.6em] &\text{eigenvalue error}\end{aligned}$},
xmin=0, xmax=42,
ymin=2.e-7, ymax=1,
ymode=log,
ylabel near ticks,
xtick = {10, 20, 30, 40},
xticklabels = {,,,,},
xmajorgrids,
x grid style={lightgray!92!black},
ytick = {1e-1, 1e-3, 1e-5},
ymajorgrids,
y grid style={lightgray!92!black}]
%
\addplot [semithick, color0, dotted, mark=*, mark size=2, mark options={solid,fill opacity=0.4}]
table {%
1 0.175172144078689
2 0.0803178018098995
3 0.0537995463048197
4 0.0329138647203931
5 0.0218875397465976
6 0.0149639433470159
7 0.00860023252116874
8 0.00547585345938661
9 0.00351426525530772
10 0.00209340790353609
11 0.00130951884713874
12 0.000820917074847774
13 0.000502433506628343
14 0.000312649933833394
15 0.000194660973449482
16 0.000121330388674719
17 7.57639867144935e-05
18 4.74530139171438e-05
19 2.98642075686684e-05
20 1.89371805445085e-05
21 1.21489106212013e-05
22 7.93183047941195e-06
23 5.31206363155383e-06
24 3.6845945508762e-06
25 2.67356739882883e-06
26 2.04549000386232e-06
27 1.65531100873634e-06
28 1.4129208366658e-06
29 1.26234111855998e-06
30 1.16879661760483e-06
31 1.11068401899076e-06
32 1.07458274493104e-06
33 1.05215555095524e-06
34 1.03822310324783e-06
35 1.02956784414356e-06
36 1.0241909337212e-06
37 1.02085063393564e-06
38 1.01877553790172e-06
39 1.01748642501025e-06
40 1.0166855876872e-06
};
\addplot [semithick, color2, dotted, mark=square*, mark size=2, mark options={solid, fill opacity=0.4}]
table {%
1 0.175172144078689
2 0.0803178018098995
3 0.0537737355996108
4 0.0412873730617694
5 0.0343492217763272
6 0.0303710037562348
7 0.0281533727127705
8 0.012678340686425
9 0.00891651777687578
10 0.0067830979633678
11 0.00548256865848802
12 0.00468072381969973
13 0.00418565908658012
14 0.00387932760607062
15 0.000890923252314097
16 0.000640144905637861
17 0.000493125594678522
18 0.000403165325528262
19 0.000347664677915481
20 0.0003133279947918
21 0.000292056065468048
22 4.4350065621468e-05
23 3.38321453635566e-05
24 2.73998815995662e-05
25 2.3421890871173e-05
26 2.09567187160714e-05
27 1.94277271967552e-05
28 1.84789400209521e-05
29 2.05693376441782e-06
30 1.66137187441648e-06
31 1.41643728884386e-06
32 1.26442667312641e-06
33 1.17004920406392e-06
34 1.11144294891034e-06
35 1.07504553939396e-06
36 1.05243911508592e-06
37 1.03839747052933e-06
38 1.02967535069728e-06
39 1.02425734858843e-06
40 1.02089172364996e-06
};
\addplot [thick, gray, dashed]
table {
1 0.105187160700136
2 0.050781778128737
3 0.0310692041456039
4 0.0195920616709266
5 0.0123513269176417
6 0.00775214078562765
7 0.00484695731517787
8 0.00302265878405023
9 0.00188200605454039
10 0.00117078178316744
11 0.000728065868195904
12 0.000452761478627043
13 0.00028165833073459
14 0.000175348947450321
15 0.000109307154858227
16 6.82833165175133e-05
17 4.28007208416957e-05
18 2.69718106163181e-05
19 1.71393335067245e-05
20 1.10315858644547e-05
21 7.23751393785852e-06
22 4.88064307495427e-06
23 3.41654246479468e-06
24 2.50702701014008e-06
25 1.94202178658898e-06
26 1.59102965999165e-06
27 1.37298560977021e-06
28 1.23753136620252e-06
29 1.15338369407816e-06
30 1.10110887991765e-06
31 1.06863430437934e-06
32 1.04846016472451e-06
33 1.035927401052e-06
34 1.02814167684092e-06
35 1.02330495235328e-06
36 1.02030023523026e-06
37 1.01843361215589e-06
38 1.01727400939331e-06
39 1.01655362798007e-06
40 1.01610610640315e-06
};
\end{axis}
\begin{axis}[
yshift=-3.6cm,
width=11.5cm,
height=5cm,
xlabel={iteration number $j$},,
ylabel={residual},
xmin=0, xmax=42,
ymin=0.5e-7, ymax=5,
ymode=log,
ylabel near ticks,
xtick = {10, 20, 30, 40},
xmajorgrids,
x grid style={lightgray!92!black},
ytick = {1e0, 1e-2, 1e-4, 1e-6},
ymajorgrids,
y grid style={lightgray!92!black}]
%
\addplot [semithick, color0, dotted, mark=*, mark size=2, mark options={solid,fill opacity=0.4}]
table {%
1 1.17416043209588
2 0.674530144579041
3 0.38269750673664
4 0.0592695473262225
5 0.046215243441817
6 0.036920495343898
7 0.00743553551530593
8 0.00600090073591261
9 0.00483287315788085
10 0.000957691600617551
11 0.000767176478598096
12 0.000612667170383346
13 0.000121349978635718
14 9.64507093875243e-05
15 7.65253985429383e-05
16 6.06307353262659e-05
17 4.79829599363156e-05
18 3.79392612214528e-05
19 2.99765327391922e-05
20 2.36718226190572e-05
21 1.86850006871819e-05
22 1.47437408757639e-05
23 1.16307664955964e-05
24 9.17319794371116e-06
25 7.23377506355218e-06
26 5.7036992816423e-06
27 4.49684302807301e-06
28 3.54509368305854e-06
29 2.79462709620958e-06
30 2.20293564512049e-06
31 1.73646448179448e-06
32 1.36873542390596e-06
33 1.07886027843639e-06
34 8.5036421896269e-07
35 6.70255443915683e-07
36 5.282901084245e-07
37 4.16391874743355e-07
38 3.28193758830743e-07
39 2.58676630670689e-07
40 2.03884050906466e-07
};
\addplot [semithick, color2, dotted, mark=square*, mark size=2, mark options={solid, fill opacity=0.4}]
table {%
1 1.17416043209588
2 0.674530144579041
3 0.38269750673664
4 0.247136452563845
5 0.176467064629849
6 0.1324190945795
7 0.101193192560847
8 0.034721523948619
9 0.0279984563513698
10 0.0218131224936807
11 0.0171440677729511
12 0.0135631401337355
13 0.010753179105889
14 0.00852338584839072
15 0.00241921286909601
16 0.00183302194952671
17 0.0014045508461
18 0.00109367255143312
19 0.000858860148839966
20 0.000677132961031907
21 0.000534731100874874
22 0.000119905126631267
23 9.32175693964593e-05
24 7.27428939919476e-05
25 5.70652847173276e-05
26 4.49080529059341e-05
27 3.53981672851416e-05
28 2.79228419025384e-05
29 5.7553928263497e-06
30 4.51322332494001e-06
31 3.54576405758959e-06
32 2.79077824119236e-06
33 2.19893151987355e-06
34 1.73355068070781e-06
35 1.36699336569284e-06
36 1.0780314739734e-06
37 8.50153047472934e-07
38 6.70420811616633e-07
39 5.28659629234278e-07
40 4.16851612707478e-07
};
\addplot [thick, gray, dashed]
table {
1 0.00210633187918105
2 0.00128463915290457
3 0.000844500700193529
4 0.000637575277742926
5 0.000513171233030634
6 0.000420227565177221
7 0.000343801694059824
8 0.000279568818297522
9 0.000225847334814165
10 0.000181418233452752
11 0.000145072547771062
12 0.000115604507637002
13 9.1877692710403e-05
14 7.28735811450501e-05
15 5.77121087616966e-05
16 4.56521074879638e-05
17 3.6080518844788e-05
18 2.849668104598e-05
19 2.24954528104974e-05
20 1.77511667403522e-05
21 1.40033189624201e-05
22 1.10442867928338e-05
23 8.70904355818688e-06
24 6.86668781814531e-06
25 5.41354735716644e-06
26 4.26761154532038e-06
27 3.3640632408865e-06
28 2.65170902982156e-06
29 2.09013664749228e-06
30 1.64745709577457e-06
31 1.2985145684739e-06
32 1.02346929211334e-06
33 8.06676812851995e-07
34 6.35802455167092e-07
35 5.01122042394893e-07
36 3.94969969093739e-07
37 3.11303732572416e-07
38 2.45360446675703e-07
39 1.93385964264679e-07
40 1.52421294672439e-07
};
\end{axis}
\end{tikzpicture}
\caption{Convergence history of the inverse power iteration for the nonlinear example. The relative error in the eigenvalue (top) and the residual estimated in the dual norm (bottom) are shown. We compare the results for~$3$ ({\protect \blueCirLight{}}) and $7$ ({\protect \greenSquLight{}}) iteration steps per mesh. The dashed line corresponds to computations on the fine mesh.}
\label{fig:nonlinear:conv}
\end{figure}
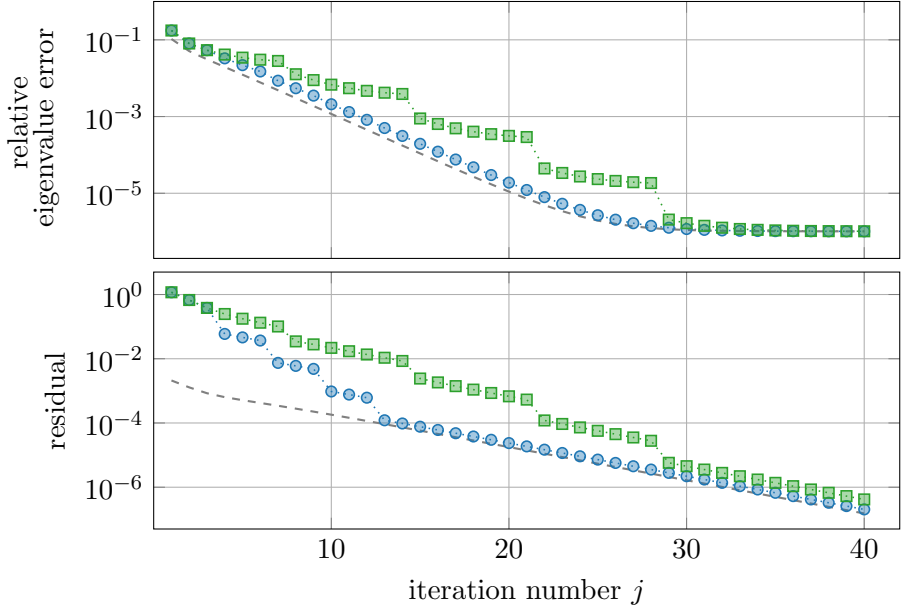
A comparison of the runtimes, showing the speed-up due to the combination of the iterative solver and mesh refinement, is shown in Table~\ref{tab:runtimes:nonlin}. 
\begin{table}	
\caption{Runtime comparison (in seconds) until the residual reaches the tolerance tol=$10^{-6}$ in the second numerical experiment.} 
\begin{tabular}{r|r|r|r}
	on fine mesh & $3$ steps per mesh & $5$ steps per mesh & $7$ steps per mesh \\[0.1em] \hline 
	100.56 & 71.96 & 51.39 & 34.49		\\
\end{tabular}
\label{tab:runtimes:nonlin}
\end{table}

Finally, the eigenmode corresponding to the first eigenvalue obtained by the inverse power method \eqref{eq:linearizedEVP:invIteration} applied to the linearized system is depicted in Figure~\ref{fig:eigenmode}. 
\begin{figure}
\begin{subfigure}{.5\textwidth}
  \centering
  \includegraphics[width=.8\linewidth]{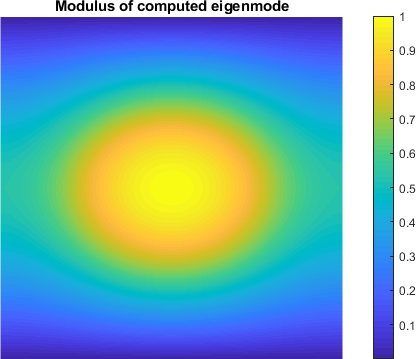}
\end{subfigure}%
\hspace{-1.15cm}
\begin{subfigure}{.5\textwidth}
  \centering
  \includegraphics[width=.8\linewidth]{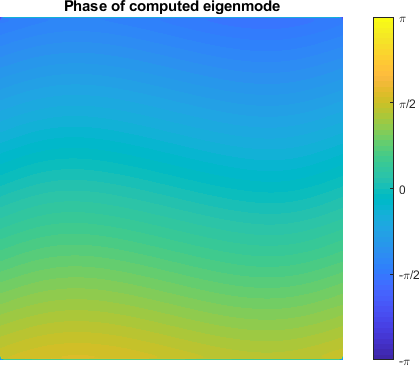}
\end{subfigure}
\caption{Illustration of the eigenmode corresponding to the first eigenvalue of the nonlinear eigenvalue problem of Section~\ref{sect:numerics:nonlin} showing its modulus (left) and phase (right).}
\label{fig:eigenmode}
\end{figure}
%
%
\subsection{Newton method for full Drude-Lorentz model}\label{sect:numerics:Newton}
In this final experiment we consider 66\% porous silicon within the disk $\Omega_2$ and model its relative permittivity with a function of the form \eqref{eq:realDrudeLorentzModel} using $\alpha_\text{DL} = 1.143$ and the material parameters given in~Table~\ref{tab:pourousSilicon}. 
\begin{table}
\caption{Material parameters representing 66\% porous silicon, used in the third numerical experiment, cf.~\cite{EffKE12}.}
\begin{tabular}{c|ccccccc}
$\ell$&1&2&3&4&5&6&7\\ \hline
$\xi_\ell^2$& 416.6166 & 352.7054&-339.9124&492.5687&-19.6143&-527.5597&98.0101
\\[0.1cm]
$\eta_\ell^2$&92.1086& 71.6269& 71.4552& 227.8301& 47.4923&  93.5605& 121.3762\\[0.1cm]
$\gamma_\ell$&2.7820& 0.9597& 0.9500& 13.1508& 9.2697& 3.2624& 2.2712\\[0.1cm]
\end{tabular}
\label{tab:pourousSilicon}
\end{table}
We apply the Newton iteration \eqref{eq:1stNewtonIt} with a randomly chosen normalizing vector $y$ to the nonlinear eigenvalue problem \eqref{eq:nonlinearEVP} for the fixed wave vector~$\k=[\pi/2, \pi]^T$. Since this eigenvalue problem depends only on the square of $\omega$, we set $\lambda\coloneqq\omega^2$ and actually consider $\lambda$ to be the eigenvalue instead of $\omega$.

In order to produce starting values $u^0$ and $\lambda^0$ for the Newton iteration, we approximate $\eps_\text{DL}^{\Re}(\omega)$ by the constant value $2$ in the disk $\Omega_2$ and apply 8 Rayleigh quotient iteration steps, similarly as in Section~\ref{sect:numerics:lin}.   

The resulting convergence history of the eigenvalue and the residual are shown in Figure \ref{fig:Newton:conv}. 
One can observe that already 6 iterations on the fine mesh (corresponding to the fourth level of refinement) are sufficient to converge to an eigenpair for which the residual is of the order of magnitude of the machine precision. 
With such a fast converging method, a mesh refinement should be applied after no more than 2 iterations per mesh. 
%
%
\begin{figure}
\begin{tikzpicture}
\begin{axis}[
width=11.5cm,
height=5cm,
ylabel={$\begin{aligned}&\qquad\text{relative}\\[-0.6em] &\text{eigenvalue error}\end{aligned}$},
xmin=0.5, xmax=12.5,
ymin=1.e-6, ymax=4,
ymode=log,
ylabel near ticks,
xtick = {5,10},
xticklabels = {,,,,},
xmajorgrids,
x grid style={lightgray!92!black},
ytick = {1e-1, 1e-3, 1e-5},
ymajorgrids,
y grid style={lightgray!92!black}]
\addplot [semithick, color0, dotted, mark=*, mark size=2, mark options={solid,fill opacity=0.4}]
table {%
1 0.180951712548264
2 0.0983219588060116
3 0.00203522834226243
4 0.0174275805461479
5 0.000933386647162685
6 0.00142269054204026
7 9.10543368458188e-05
8 9.43366276253895e-05
9 5.6564233954108e-06
10 5.67104064360072e-06
11 5.67104065046846e-06
12 5.67104064247946e-06
};

\addplot [thick, gray, dashed]
table {
1 0.341088586628126
2 0.0920806076627799
3 0.0112870218589273
4 0.000210752987189773
5 5.59516748535875e-06
6 5.67104065046846e-06
7 5.6710406423393e-06
8 5.67104064962751e-06
9 5.67104064051724e-06
10 5.67104064948736e-06
11 5.67104063673298e-06
12 5.67104064402119e-06
};
\end{axis}
\begin{axis}[
yshift=-3.6cm,
width=11.5cm,
height=5cm,
xlabel={iteration number $j$},,
ylabel={residual},
xmin=0.5, xmax=12.5,
ymin=1e-19, ymax=150,
ymode=log,
ylabel near ticks,
xtick = {5,10},
xmajorgrids,
x grid style={lightgray!92!black},
ytick = {1e0, 1e-8, 1e-16},
ymajorgrids,
y grid style={lightgray!92!black}]
%
\addplot [semithick, color0, dotted, mark=*, mark size=2, mark options={solid,fill opacity=0.4}]
table {%
1 0.254982238392878
2 0.0230080676451307
3 0.00491370046610074
4 0.000129307829872163
5 2.72699674969724e-05
6 2.24514216454097e-08
7 4.27469906784095e-08
8 2.34375300871286e-13
9 4.68723870633866e-11
10 3.63567257997596e-17
11 3.34663843421674e-17
12 2.04432277376356e-17
};
%
\addplot [thick, gray, dashed]
table {
1 0.000830895133544129
2 0.000255989103210334
3 3.17341973054715e-05
4 6.04816901055567e-07
5 2.30214202325759e-10
6 8.07890212715645e-17
7 1.80105960158348e-17
8 2.37213223778371e-17
9 3.09047694491167e-17
10 4.14491294588821e-17
11 2.67822773381992e-17
12 2.08688416929541e-17
};
\end{axis}
\end{tikzpicture}
\caption{Convergence history for the Newton method. The relative error in the eigenvalue (top) and the residual measured in the dual norm (bottom) are shown for~$2$ ({\protect \blueCirLight{}}) iteration steps per mesh and on the fine mesh (dashed line).}
\label{fig:Newton:conv}
\end{figure}
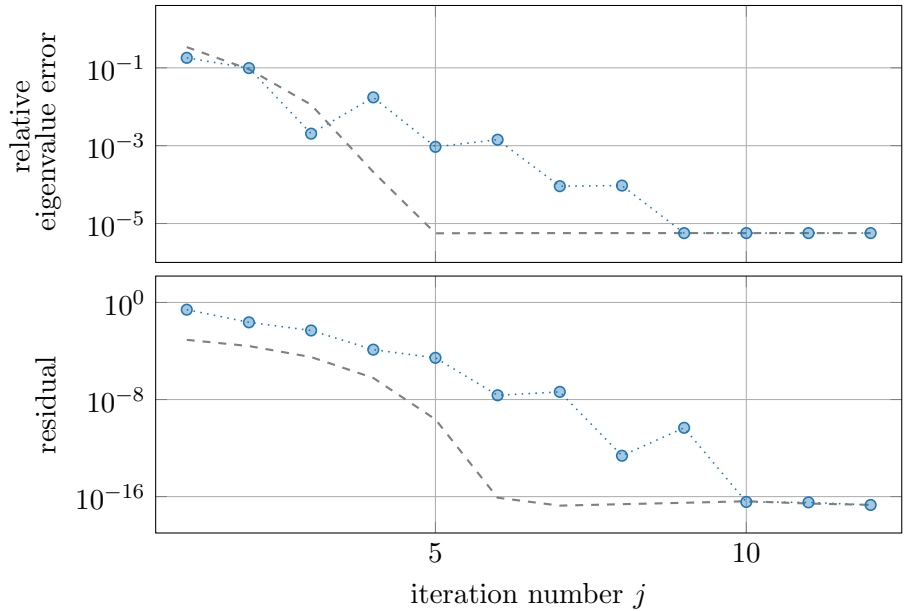		    
%
%
\section{Conclusion}\label{sect:conclusion}
In this paper, we have considered iterative methods for linear Hermitian as well as specific nonlinear eigenvalue problems arising in photonic crystal modeling. 
In case the electric permittivity is given by a Drude-Lorentz model with real coefficients and no dissipation, we are able to linearize the problem to obtain a linear and Hermitian eigenvalue problem. For this, we show the convergence of the inverse power method. 

For more realistic models taking dissipation into account, the same procedure would lead to a linear but non-Hermitian eigenvalue problem. Thus, instead of a linearization we directly apply Newton's method, for which we prove local convergence. 
%
%
\section*{Acknowledgements}
The authors thank Christoph Zimmer (TU Berlin) for the helpful discussions on the paper.
%
%
\newcommand{\etalchar}[1]{$^{#1}$}

\end{document}